\documentclass[a4paper,11pt,twoside]{amsart}
\usepackage[english]{babel}
\usepackage[utf8]{inputenc}

\usepackage[a4paper,inner=3cm,outer=3cm,top=4cm,bottom=4cm,pdftex]{geometry}
\usepackage{fancyhdr}
\usepackage{fancyhdr}
\usepackage{titlesec}
\usepackage{enumerate}
\usepackage{color}
\usepackage{bold-extra}
\titleformat{\section}{\normalfont\scshape\centering}{\thesection}{1em}{}
\titleformat{\subsection}{\bfseries}{\thesubsection}{1em}{}

\usepackage{color}
\usepackage{bold-extra}
\usepackage{ mathrsfs }

\usepackage{comment}
\usepackage{graphics}
\usepackage{aliascnt}
\usepackage[pdftex,citecolor=green,linkcolor=red]{hyperref}

\usepackage{amsmath}
\usepackage{amsfonts}
\usepackage{amssymb}
\usepackage{amsthm}
\usepackage{comment}
\usepackage{mathtools}

\newtheorem{theorem}{Theorem}[section]

\newtheorem{lemma}[theorem]{Lemma}
\newtheorem{proposition}[theorem]{Proposition}
\theoremstyle{definition}

\newtheorem{remark}[theorem]{Remark}

\numberwithin{equation}{section}

\renewcommand\d{\textnormal{d}}

\begin{document}

\title{Almost primes in almost all short intervals II}

\author{Kaisa Matom\"aki}
\address{Department of Mathematics and Statistics, University of Turku, 20014 Turku, Finland}
\email{ksmato@utu.fi}

\author{Joni Ter\"{a}v\"{a}inen}
\address{Department of Mathematics and Statistics, University of Turku, 20014 Turku, Finland}
\email{joni.p.teravainen@gmail.com}


\begin{abstract} We show that, for almost all $x$, the interval $(x, x+(\log x)^{2.1}]$ contains products of exactly two primes. This improves on a work of the second author that had $3.51$ in place of $2.1$. To obtain this improvement, we prove a new type II estimate. One of the new innovations is to use Heath-Brown's mean value theorem for sparse Dirichlet polynomials.
\end{abstract}

\maketitle

\section{Introduction}

We shall study the distribution of $E_2$ numbers, i.e. numbers with exactly two prime factors, in almost all short intervals. This problem has been studied in previous works of Heath-Brown~\cite{hb-1978}, Motohashi~\cite{motohashi}, Wolke~\cite{wolke}, Harman~\cite{harman-almostprimes}, and the second author~\cite{tera-primes}. 

The best known result~\cite[Theorem 3]{tera-primes} gives that, for almost all $x$, the interval $(x,x+(\log x)^{3.51}]$ contains $E_2$ numbers (here, and in the rest of the paper, we say that a property $P(x)$ holds for almost all $x$ if the measure of $x\in [1,X]$ for which $P(x)$ fails is $o(X)$ as $X\to \infty$). In this paper, we strengthen this result by replacing the exponent $3.51$ by $2.1$. In the theorem and later, $p_j$ always denotes a prime.

\begin{theorem}\label{thm_main}
There exist constants $c_0 > 0$ and $\delta>0$ such that the following holds. Let $X\geq 3$. Then, for all but $\ll X/(\log X)^{\delta}$ integers $x \in [2, X]$, we have
\begin{align*}
|\{p_1 p_2 \in (x,x+(\log x)^{2.1}] \colon (\log x)^{1.09} < p_1 \leq (\log x)^{1.1}\}| \geq c_0 (\log x)^{1.1}.    
\end{align*}
\end{theorem}

One can show that the lower bound in Theorem~\ref{thm_main} is of the correct order of magnitude, although it is only for those $E_2$ numbers that have a prime factor in a certain superdyadic interval. 

We remark that the limit of the approaches in~\cite{harman-almostprimes, tera-primes} was the exponent $3+\varepsilon$, which could be reached in~\cite{harman-almostprimes} conditionally assuming the following slight strengthening of the density hypothesis: For any $\varepsilon > 0$, there exists $\delta = \delta(\varepsilon) > 0$ such that, for any $\sigma \in [1/2+\varepsilon, 1]$ and $T \geq 3$, one has
\begin{equation}
\label{eq:densHyp}
N(\sigma, T) \ll T^{(2-\delta)(1-\sigma) + o(1)},
\end{equation}
where $N(\sigma, T)$ is the number of zeros of the Riemann zeta-function in the rectangle $\{b+it \colon b \geq \sigma, |t| \leq T\}$.

On the other hand, a result of Selberg~\cite{selberg} from 1943 shows that under the Riemann hypothesis almost all intervals $(x,x+(\log x)^{2+\varepsilon}]$ contain primes, and this easily implies that almost all such intervals contain $E_2$ numbers as well (since if $p\in (x/2,x/2+(\log x)^{2+\varepsilon}/2]$ is a prime, then $2p\in (x,x+(\log x)^{2+\varepsilon}]$ is an $E_2$ number). Theorem~\ref{thm_main} gets somewhat close to the exponent of $2+\varepsilon$, which seems to be the barrier for $E_2$ numbers even under the Riemann hypothesis. In fact, as discussed in Section~\ref{se:H-B}, in order to obtain $2+\varepsilon$ for $E_2$ numbers, it suffices to assume the Lindel\"of hypothesis which of course is a weaker assumption than the Riemann hypothesis. Actually we believe that the above variant of the density hypothesis is a sufficient assumption for obtaining $2+\varepsilon$ but we plan to return to this on a later occasion.

In addition to $E_2$ numbers, also $P_2$ numbers that have at most two prime factors are called almost primes. The question of short interval distribution for these is significantly easier since classical sieve methods are applicable. Indeed, the first author~\cite{matomaki-P2} has recently shown that, for almost all $x$, the interval $(x, x+h(x)\log x]$ contains $P_2$ numbers, provided only that $h(x) \to \infty$ as $x \to \infty$.

For $E_3$ numbers, i.e. numbers with exactly three prime factors, much shorter intervals can be reached than for $E_2$ numbers; the second author showed in~\cite{tera-primes} that, for almost all $x$, the interval $(x, x+(\log x)(\log \log x)^{6+\varepsilon}]$ contains $E_3$-numbers. On the other hand, for the primes, the best known result due to Jia~\cite{jia} gives that, for almost all $x$, the interval $(x, x+x^{1/20}]$ contains primes. Hence we understand the short interval distribution of $E_k$ numbers for $k \geq 2$ significantly better than that of the primes.

We lastly note that the same method that we apply for $E_2$ numbers in almost all intervals readily adapts to $E_3$ numbers in \emph{all} intervals. Indeed, following the proof of Theorem~\ref{thm_main} very closely, we obtain in Section~\ref{sec:allintervals} the following. 

\begin{theorem}\label{thm_all} For all large enough $x$, the interval $(x,x+\sqrt{x}(\log x)^{1.55}]$ contains $E_3$ numbers. 
\end{theorem}

In comparison, for $E_2$ numbers in all intervals, we are not aware of results that would go below the interval length $\asymp x^{0.525}$ known for the primes (and consequently for $E_2$ numbers) by the work of Baker, Harman and Pintz~\cite{baker-harman-pintz}.

As far as we are aware, Theorem~\ref{thm_all} is the first result on $E_3$ numbers in all intervals of length $\sqrt{x}(\log x)^c$. It would be possible to similarly adapt also earlier works on $E_2$ numbers in almost all short intervals, such as~\cite{tera-primes} or \cite{harman-almostprimes}, to produce a result of this shape, but with a larger value of $c$.

\subsection{Proof ideas}
The beginning of our argument follows~\cite{tera-primes} with some simplifications. In particular, we first apply Harman's sieve to find a suitable minorant $\rho^-(n) \leq 1_{\mathbb{P}}(n)$ and then, by a standard application of Perron's formula, reduce matters to mean squares of Dirichlet polynomials. Once we have made this reduction, we need to prove that, for some $\varepsilon > 0$,
\begin{align*}
\int_{X^{1/1000}}^{X/h}|P_1(1+it)|^2|P(1+it)|^2\d t\ll \frac{1}{(\log X)^{2+\varepsilon}},
\end{align*}
where $h = (\log X)^{2.1}$,
\begin{align*}
P_1(s):=\sum_{\substack{p_1 \sim P_1}} \frac{1}{p_1^{s}},\quad P(s):=\sum_{X/(2P_1)\leq n\leq  4X/P_1}\frac{\rho^-(n)}{n^{s}},
\end{align*}
and $P_1 = (\log X)^{1.1}$ (as well as other very similar claims).

Still following~\cite{tera-primes}, we partition $[X^{1/1000}, X/h] = \mathcal{T} \cup \mathcal{U}$ according to the size of $P_1(s)$, with 
\[
\mathcal{T} := \{t \in [X^{1/1000}, X/h] \colon |P_1(1+it)| \leq P_1^{-\varepsilon}\}.
\]
The integral over $\mathcal{T}$ is easily dealt with in the beginning of Section~\ref{sec:proof} --- we use the pointwise bound $|P_1(1+it)| \leq P_1^{-\varepsilon}$ and estimate the mean square of $P(1+it)$ using (an improved) mean value theorem (Lemma~\ref{le:contMVT2} below).

Let us turn to the integral over $\mathcal{U}$. The minorant $\rho^-(n)$ is chosen so that it can be split into appropriate type I, type I/II, and type II sums (see Proposition~\ref{prop_rho}). We deal with type I and type I/II sums in Sections~\ref{subsec:typeI} and~\ref{subsec:typeI/II} in a rather similar manner as in~\cite{tera-primes}, utilizing mean value theorems of Watt~\cite{watt} and Deshouillers--Iwaniec~\cite{DI} (see Lemma~\ref{le:WattDI}).

The most novel part of our argument is the treatment of our type II sums which lead to integrals of the type
\begin{align*}
\int_{\mathcal{U}} |P_1(1+it)|^2|M_1(1+it)|^2 |M_2(1+it)|^2\d t,
\end{align*}
where, for some coefficients $\alpha_m, \beta_n$,
\[
M_1(s) = \sum_{m \sim M_1} \frac{\alpha_m}{m^s} \quad \text{and} \quad M_2(s) = \sum_{n \asymp X/(P_1M_1)} \frac{\beta_n}{n^s} 
\]
with $M_1 \in [X^{\varepsilon/2}, X^{2/11}]$. We further split $\mathcal{U}$ into sets $\mathcal{U}_{\sigma_1, \sigma_2}$, where
\[
\mathcal{U}_{\sigma_1, \sigma_2} := \{t \in \mathcal{U} \colon |M_1(1+it)| \in (M_1^{-\sigma_1}, 2M_1^{-\sigma_1}], |M_2(1+it)| \in (M_2^{-\sigma_2}, 2M_2^{-\sigma_2}]\}.
\]
Now it suffices to show that, for any $\sigma_1, \sigma_2$,
\[
|\mathcal{U}_{\sigma_1, \sigma_2}| \ll_A \frac{M_1^{\sigma_1} M_2^{\sigma_2}}{(\log X)^A}.
\]

When $\sigma_j \leq 49/206 - 10\varepsilon$ for $j = 1$ or $j = 2$, we are able to use Jutila's~\cite{jutila-large-value} large value estimate to obtain a satisfactory bound (see Proposition~\ref{prop_typeII}).

We deal with the case $\sigma_1, \sigma_2 > 49/206 - 10\varepsilon$ in Section~\ref{subsec:typeII}. There (utilizing an idea from~\cite{matomaki-radziwill}) we use the definitions of $\mathcal{U}_{\sigma_1, \sigma_2}$ and $\mathcal{U}$ to see that 
\begin{align*}
|\mathcal{U}_{\sigma_1, \sigma_2}| &\leq M_2^{2\sigma_2} P_1^{2k\varepsilon} \int_{\mathcal{U}_{\sigma_1, \sigma_2}}|P_1(1+it)|^{2k} |M_2(1+it)|^2 \d t \\
&\leq M_2^{2\sigma_2} P_1^{2k\varepsilon} \int_{X^{1/1000}}^{X/h} |P_1(1+it)|^{2k} |M_2(1+it)|^2\d t,
\end{align*}
where we have chosen $k$ so that $P_1^k=X^{1-o(1)}$. Now the coefficients of $P_1(s)^k$ are supported on $P_1 = (\log X)^{1.1}$-smooth numbers, so they have a very sparse support (of size $X^{1-1/1.1+o(1)}$ by standard estimates on smooth numbers). At this point, we invoke a mean value theorem for sparse Dirichlet polynomials proven recently by Heath-Brown~\cite{hb-smooth} (see Lemma~\ref{le_HB_MVT} below). This leads to a satisfactory bound unless $M_1 \in [X^{103/594}, X^{2/11}]$ and $\sigma_1 > 1/2-(\log X)/(22(\log M_1))$ (see~\eqref{eq:thsig1} with $a = 1.1$ and $\theta=(\log M_1)/(\log X)$). In the remaining range we argue similarly but use
\begin{align*}
|\mathcal{U}_{\sigma_1, \sigma_2}| &\leq M_1^{10 \sigma_1} P_1^{2k\varepsilon} \int_{\mathcal{U}_{\sigma_1, \sigma_2}}|P_1(1+it)|^{2k} |M_1(1+it)|^{10} \d t.
\end{align*}
\subsection{Acknowledgments}

KM was supported by Academy of Finland grant no. 285894. JT was supported by Academy of Finland grant no. 340098. We thank the referee for numerous helpful suggestions. 

\subsection{Notation}
We use the usual asymptotic notation $\ll, \gg, \asymp, O(\cdot), o(\cdot)$ and use $n\sim y$ as a shorthand for $y<n\leq 2y$. The letters $p, q,$ and $p_j, q_j$ will always denote primes.

For a claim $A$, we write $1_A$ for its indicator function and for a set $A$ we write $1_A(n) = 1_{n \in A}$. For $z \geq 2$, we write $P(z):=\prod_{p<z}p$ and $\rho(n,z) = 1_{(n, P(z)) = 1}$. In particular, Buchstab's identity states that, for any $z > w \geq 2$, we have
\begin{equation}
\label{eq:Buchstab}
\rho(n, z) = \rho(n, w) - \sum_{\substack{n = qm \\ w \leq q < z}} \rho(n/q, q).
\end{equation}

We denote by $\mu$ the M\"obius function and by $d_k$ the $k$-fold divisor function, and denote $d_2(n)$ simply by $d(n)$. We will use occasionally the fact that $d_k$ satisfies the submultiplicativity property $d_k(mn)\leq d_k(m)d_k(n)$ for all $m,n\in \mathbb{N}$. We say that a sequence $(\alpha(n))_{n\sim N}$ is \emph{divisor-bounded} if $|\alpha(n)|\ll d(n)^{B}$ for some fixed $B$. Note that if $(\alpha(n))_{n\sim M}$ and $(\beta(n))_{n\sim N}$ are divisor-bounded, then $(\alpha*\beta(n))_{n\asymp MN}$ is clearly also divisor-bounded. 

For any multiplicative function $f \colon \mathbb{N} \to [1, \infty)$ we have, by writing $f = 1 \ast g \iff g = f \ast \mu \geq 0$, the elementary upper bound
\[
\sum_{n \leq x} f(n) = \sum_{m \leq x} g(m)\left(\frac{x}{m} + O(1)\right) \ll x \prod_{p \leq x} \left(1+\frac{g(p)}{p} + \frac{g(p^2)}{p^2} + \dotsb\right).
\]
In particular this together with Mertens' formula implies that, for any fixed $j, k , c, d \geq 1$,
\begin{equation}
\label{eq:divbound}
\sum_{n \leq X} d_{j}(n)^{c} d_k(n)^d \ll \prod_{p \leq X} \left(1+\frac{j^c k^d - 1}{p} + O\left(\frac{1}{p^{3/2}}\right)\right) \ll X (\log X)^{j^c \cdot k^d-1}.
\end{equation}

\section{The minorant function}
\label{sec:minorant}
In this section we first construct our minorant function $\rho^-(n) \leq 1_{\mathbb{P}}(n)$ using Harman's sieve method~\cite{harman-sieves}. Then in Subsection~\ref{ssec:minave} we show that it has positive average over long intervals and in Subsection~\ref{ssec:mintypeIII} we show that it can be decomposed into appropriate type I, type I/II, and type II sums.

For the construction, recall that $\rho(n, z) = 1_{(n, P(z)) = 1}$. Let $n \in [2X^{1/2}, 3X]$, $z = X^{2/11}$, and let $\varepsilon>0$ be small. Applying Buchstab's identity~\eqref{eq:Buchstab} twice we obtain
\begin{align}
\begin{split}
\label{eq:Buchstab1}
1_{n \in \mathbb{P}} &= \rho(n, 2X^{1/2}) = \rho(n, z) - \sum_{\substack{n = qm \\ z \leq q < 2X^{1/2}}} \rho(m, z) \\
&\quad +\sum_{\substack{n = q_1 q_2 m \\ z \leq q_2 < q_1 < X^{1/4-2\varepsilon} \\ q_1 q_2^4 < X^{1-2\varepsilon}}} \rho(m, q_2) + \sum_{\substack{n = q_1 q_2 m \\ z \leq q_2 < q_1 < 2X^{1/2} \\ q_1 \geq X^{1/4-2\varepsilon} \text{ or } q_1 q_2^4 \geq X^{1-2\varepsilon}}} \rho(m, q_2).
\end{split}
\end{align}
Applying Buchstab's identity twice more, the third term on the right-hand side equals
\begin{equation}
\label{eq:AfterSecondBuch}
\sum_{\substack{n = q_1 q_2 m \\ z \leq q_2 < q_1 < X^{1/4-2\varepsilon} \\ q_1 q_2^4 < X^{1-2\varepsilon}}} \rho(m, z) - \sum_{\substack{n = q_1 q_2 q_3 m \\ z \leq q_3 < q_2 < q_1 < X^{1/4-2\varepsilon} \\ q_1 q_2^4 < X^{1-2\varepsilon}}} \rho(m, z) + \sum_{\substack{n = q_1 q_2 q_3 q_4 m \\ z \leq q_4 < q_3 < q_2 < q_1 < X^{1/4-2\varepsilon} \\ q_1 q_2^4 < X^{1-2\varepsilon}}} \rho(m, q_4).
\end{equation}
We define our minorant for $\rho(n, 2X^{1/2})$ by discarding the last term here as well as the last term on the right-hand side of~\eqref{eq:Buchstab1} (both terms are nonnegative, so they can be discarded when we look for a minorant), and thus choose
\begin{align}
\label{eq:rho-def}
\rho^-(n) &:= \rho(n, z) - \sum_{\substack{n = qm \\ z \leq q < 2X^{1/2}}} \rho(m, z) +  \sum_{\substack{n = q_1 q_2 m \\ z \leq q_2 < q_1 < X^{1/4-2\varepsilon} \\ q_1 q_2^4 < X^{1-2\varepsilon}}} \rho(m, z) - \sum_{\substack{n = q_1 q_2 q_3 m \\ z \leq q_3 < q_2 < q_1 < X^{1/4-2\varepsilon} \\ q_1 q_2^4 < X^{1-2\varepsilon}}} \rho(m, z).
\end{align}
We note here for later use that, since $n$ has $\leq (\log (3X))/(\log z)$ prime factors that are $\geq z$, we have the bound
\begin{align}\label{eq:rhobound}
|\rho^{-}(n)|\leq 4 \left(\frac{\log (3X)}{\log z}\right)^3\rho(n,z) \ll \rho(n, z).
\end{align}

Theorem~\ref{thm_main} will follow from the following variance estimate in short intervals.
\begin{theorem}\label{thm_minorant}
Let $\varepsilon > 0$ be sufficiently small, let $X\geq 3, h=(\log X)^{c}$ with $c=2.1$, and $h_1=X^{99/100}$. Let
\begin{equation}
\label{eq:aP1def}
a \in [c-1-1/10000, c-1], \quad \text{and} \quad P_1 = (\log X)^a. 
\end{equation}

The function $\rho^-(n)$ defined in~\eqref{eq:rho-def} with $z = X^{2/11}$ satisfies the following three conditions.
\begin{enumerate}[(i)]
	\item For every $n \in [2X^{1/2}, 3X]$, we have
\begin{align*}
\rho^-(n)\leq 1_{\mathbb{P}}(n).    
\end{align*}    
    \item Once $X$ is large enough we have, for all $x \in (X, 2X]$,
    \begin{align*}
    \sum_{\substack{x < p_1 n\leq x+h_1\\ p_1 \sim P_1}}\rho^-(n)\geq \frac{h_1}{200\log P_1 \log X}.    
    \end{align*}
\item We have
    \begin{align}
    \label{eq:minMSbound}
    \frac{1}{X}\int_{X}^{2X}\left|\frac{1}{h}\sum_{\substack{x < p_1 n\leq x+h\\ p_1 \sim P_1}}\rho^-(n)-\frac{1}{h_1}\sum_{\substack{x < p_1 n\leq x+h_1\\p_1 \sim P_1}}\rho^-(n)\right|^2 \d x \ll\frac{1}{(\log X)^{2+\varepsilon}}.
    \end{align}
\end{enumerate}
\end{theorem}

Note that property (i) is immediate from the construction above. We also remark that a bound of $\ll 1/(\log X \log P_1)^2$ for the left-hand side of~\eqref{eq:minMSbound} would be easy to prove; crucially, we must beat this bound. 

Let us first see how Theorem~\ref{thm_minorant} implies Theorem~\ref{thm_main}.

\begin{proof}[Proof of Theorem~\ref{thm_main} assuming Theorem~\ref{thm_minorant}] 
Let $\rho^-(n)$, $c$ and $\varepsilon$ be as in Theorem~\ref{thm_minorant}. Let $b_1=c-1-1/10000$, $b_2=c-1-1/20000$. Summing over different choices of $P_1$, Theorem~\ref{thm_minorant}(ii) implies that there is a  constant $\gamma>0$ such that 
  \begin{align}
  \label{eq:rho-low}
    \frac{1}{h_1}\sum_{\substack{x < p_1n\leq x+h_1 \\ (\log X)^{b_1} < p_1\leq (\log X)^{b_2}}}\rho^-(n)\geq  \frac{\gamma}{\log X}  
  \end{align}
    for all $x\in (X, 2X]$. On the other hand, Theorem~\ref{thm_minorant}(ii) and the Cauchy--Schwarz inequality imply that 
    \begin{align}
    \label{eq:minMSbound2}
   \int_{X}^{2X}\left|\frac{1}{h}\sum_{\substack{x < p_1 n\leq x+h\\ (\log X)^{b_1}<p_1\leq (\log X)^{b_2}}}\rho^-(n)-\frac{1}{h_1}\sum_{\substack{x < p_1 n\leq x+h_1\\(\log X)^{b_1}<p_1\leq (\log X)^{b_2}}}\rho^-(n)\right|^2 \d x \ll\frac{X}{(\log X)^{2+\varepsilon/2}}.
    \end{align}

    Define the exceptional set
    \begin{align*}
    \mathcal{X}:=\left\{x\geq 2 \colon\, \frac{1}{(\log x)^c}\sum_{\substack{x < p_1n\leq x+(\log x)^c\\(\log x)^{c-1-1/1000}<p_1\leq (\log x)^{c-1}}}1_{\mathbb{P}}(n)< \frac{\gamma}{2(\log x)}\right\}\cap \mathbb{N}. \end{align*}
Using the inequality $1_{\mathbb{P}}(n)\geq \rho^{-}(n)$ and combining~\eqref{eq:rho-low} and~\eqref{eq:minMSbound2}, we see that 
 \begin{align*}
    |\mathcal{X}\cap (X,2X]| \ll \frac{X}{(\log X)^{\varepsilon/2}}.
    \end{align*}
     The claim now follows by summing over dyadic intervals.
\end{proof}

\subsection{Proof of Theorem~\ref{thm_minorant}(ii)}
\label{ssec:minave}
Let us now show that the minorant $\rho^{-}$ satisfies condition (ii) of Theorem~\ref{thm_minorant}. Let $x \in (X, 2X]$ and let $p_1\sim P_1$ be a prime. Write $y = x/p_1$ and $h_2 = h_1/p_1$.

Recall that $\rho^-$ was constructed by discarding the last terms in~\eqref{eq:Buchstab1} and~\eqref{eq:AfterSecondBuch}. Hence
\begin{align*}
\sum_{y < n \leq y + h_2} \rho^-(n) = \sum_{y < q \leq y+h_2} 1 - \sum_{\substack{y < q_1 q_2 m \leq y+h_2 \\ z \leq q_2 < q_1 < 2X^{1/2} \\ q_1 \geq X^{1/4-2\varepsilon} \text{ or } q_1 q_2^4 \geq X^{1-2\varepsilon}}} \rho(m, q_2) - \sum_{\substack{y < q_1 q_2 q_3 q_4 m \leq y+h_2 \\ z \leq q_4 < q_3 < q_2 < q_1 < X^{1/4-2\varepsilon} \\ q_1 q_2^4 < X^{1-2\varepsilon}}} \rho(m, q_4).
\end{align*}
These sums can be transformed into integrals involving Buchstab's function $\omega$ (defined by $\omega(u)=1/u$ for $1\leq u\leq 2$ and extended by the delay differential equation $\frac{\d}{\d u}(u\omega(u))=\omega(u-1)$ for $u\geq 2$) using the prime number theorem (see e.g.~\cite[Lemma 16]{tera-primes} or~\cite[Section 1.4]{harman-sieves} --- the fact that we work over an interval of length $h_2 = X^{99/100}/p_1$ makes no difference since the prime number theorem in short intervals is applicable), and we obtain
\begin{align*}
&\sum_{y < n \leq y + h_2} \rho^-(n) = \frac{h_2}{\log y} \left(1-\int_{2/11}^{1/2} \int_{2/11}^{\alpha_1} 1_{\alpha_1 \geq 1/4-2\varepsilon \text{ or } \alpha_1+4\alpha_2\geq 1-2\varepsilon} \frac{\omega\left(\frac{1-\alpha_1-\alpha_2}{\alpha_2}\right)}{\alpha_1 \alpha_2^2}\d \alpha_2\d \alpha_1\right.\\
&\left.-\int_{2/11}^{1/4-2\varepsilon} \int_{2/11}^{\alpha_1} \int_{2/11}^{\alpha_2} \int_{2/11}^{\alpha_3}1_{\alpha_1+4\alpha_2 \leq 1-2\varepsilon}\frac{\omega\left(\frac{1-\alpha_1-\alpha_2-\alpha_3-\alpha_4}{\alpha_4}\right)}{\alpha_1\alpha_2\alpha_3\alpha_4^2}\d \alpha_4\d \alpha_3\d \alpha_2\d \alpha_1 + o(1)\right) \\
&=: \frac{h_2}{\log y}(1-I_2(\varepsilon)-I_4(\varepsilon)+o(1)),
\end{align*}
say. Note that the integrand in $I_4(\varepsilon)$ vanishes unless $\alpha_2 \leq 1/5$. Hence
\[
I_4(\varepsilon) \leq \left(\frac{11}{2}\right)^5 \int_{2/11}^{1/4} \int_{2/11}^{1/5} \int_{2/11}^{\alpha_2} \int_{2/11}^{\alpha_3} \d \alpha_4\d \alpha_3\d \alpha_2\d \alpha_1 \leq \left(\frac{11}{2}\right)^5 \left(\frac{1}{4}-\frac{2}{11}\right) \frac{\left(\frac{1}{5}-\frac{2}{11}\right)^3}{3!} < 0.0004.
\] 
Furthermore, a numerical calculation\footnote{The Mathematica code for computing the integral can be found along with the arXiv submission of this paper.} shows that $I_2(0)\leq 0.99$, so that, taking $\varepsilon>0$ small enough (and using continuity in $\varepsilon$) we have for any $x \sim X$ with $X$ large,
\begin{align*}
    \sum_{\substack{x < p_1n\leq x+h_1\\ p_1 \sim P_1}}\rho^-(n)\geq 0.009h_1 \sum_{p_1 \sim P_1} \frac{1}{p_1\log(X/p_1)} \geq \frac{h_1}{200\log P_1 \log X}, 
\end{align*}
where we used Mertens' theorem to obtain the last inequality.

\subsection{Reduction to type I, type I/II, and type II sums}
\label{ssec:mintypeIII}
Most of the rest of the paper is devoted to showing that the function $\rho^-$ constructed above satisfies Theorem~\ref{thm_minorant}(iii). We fix once and for all
\begin{align}\label{eq:delta}
\delta_A := (\log X)^{-10A}
\end{align}
and shall, for technical reasons, restrict many variables into $\delta_A$-adic intervals $(N, (1+\delta_A)N]$. The following proposition gives a decomposition of $\rho^-$ into convenient type I, type I/II, and type II terms.

\begin{proposition}[Decomposition of the minorant]\label{prop_rho} Let $\varepsilon > 0$ be fixed and small enough, and let $A \geq 5$. Let $X \geq 3,$ $z_0=\exp((\log X)/(\log \log X)^3)$, and $z = X^{2/11}$. Let $\rho^-$ be as in~\eqref{eq:rho-def}. Let $Y \in (X^{1-\varepsilon/100}, X/2]$. Then there exists a set $\mathcal{F}$ (depending on $X$ and $Y$) consisting of  $O(\exp((\log \log X)^5))$ functions $f \colon \mathbb{N} \to \mathbb{C}$ such that
\begin{align*}
\rho^-(n)1_{n \in (Y/2, 4Y]} = \sum_{f \in \mathcal{F}} f(n) + c_n,
\end{align*}
where $c_n$ are supported on $(Y/4, 8Y]$ and satisfy
\begin{equation}
\label{eq:cnbound}
\sum_{n} |c_n|^2 \ll_A \frac{Y}{\log^{A} X}
\end{equation}

Furthermore, for each $f \in \mathcal{F}$ one of the following holds for some divisor-bounded coefficients $\alpha_m, \beta_m$:
\begin{enumerate}[(i)]
\item (Type I case)
\begin{align*}
f(n) = \sum_{\substack{n=m_1m_2\\ m_1 \sim M_1 \\ M_2 < m_2 \leq (1+\delta_A) M_2}}\alpha_{m_1}  
\end{align*}
with $M_1 \leq X^{1/2+\varepsilon}$ and $M_1 M_2 \in (Y/2, 4Y]$.
\item (Type I/II case)
\begin{align*}
f(n) = \sum_{\substack{n=m_1m_2m_3\\ m_1 \sim M_1,\, m_2 \sim M_2 \\ M_3 < m_3 \leq (1+\delta_A) M_3}} \alpha_{m_1} \beta_{m_2}  
\end{align*}
with
\begin{align*}
M_1^2M_2\leq X^{1-\varepsilon},\quad M_2\leq X^{1/4-\varepsilon}, \quad \text{and} \quad M_1 M_2 M_3 \in (Y/2, 4Y].
\end{align*}
\item (Type II case) For some $R \in \{1, \dotsc, \lfloor \frac{\log z}{\log z_0}\rfloor\}$,
\begin{align*}
f(n) = \sum_{\substack{n=m_1m_2\\M_1 < m_1 \leq (1+\delta_A)^{R} M_1 \\ m_2\sim M_2}} \alpha_{m_1} \beta_{m_2}  
\end{align*}
with
\begin{align*}
X^{\varepsilon/2} \leq M_1\leq z, \quad M_1 M_2 \in (Y/2, 4Y],
\end{align*}
and
\begin{equation}
\label{eq:alm1def}
\alpha_m = \sum_{\substack{m = q_1 \dotsm q_R \\ Q_j < q_j \leq Q_j(1+\delta_A)}} 1,
\end{equation}
where $Q_j \in [z_0, z)$ and $Q_1 \dotsm Q_R = M_1$.
\end{enumerate}
\end{proposition}

Proposition~\ref{prop_rho} will quickly follow from the following lemma.
\begin{lemma}
\label{le:HarmanSieve}
Let $A \geq 5$. Let $X \geq 3,$ $z_0=\exp((\log X)/(\log \log X)^3)$, $z_0<z_1\leq X^{1/3}$, and $D = \exp((\log X)/(\log \log X))$. 
\begin{enumerate}[(i)] 
\item Let $\alpha_m$ be bounded. Then, for any $n \in (X^{99/100}, 4X]$, we have
\begin{align}
\label{eq:HarmanDecomposition}
\sum_{n = mk} \alpha_m \rho(k, z_1) = \sum_{\substack{n = emk \\ e \mid P(z_0),\, e \leq D}} \alpha_m \mu(e) - \sum_{\substack{Q = (1+\delta_A)^j \\ z_0 \leq Q < z_1}} \sum_{\substack{n = qmk \\ Q < q \leq Q (1+\delta_A)}} \alpha_m \rho(k, Q) + c_n,
\end{align}
where $c_n$ are such that, for any $Y \in (X^{99/100}, 4X]$, one has
\begin{equation*}
\sum_{n \sim Y} |c_n|^2 \ll_A \frac{Y}{\log^{3A} X}.
\end{equation*}
\item Let $\alpha_m$ be bounded and supported on $(m, P(z_1)) = 1$, and let $L = \lfloor \frac{\log (4X)}{\log z_0} \rfloor$.  Then, for any $n \in (X^{99/100}, 4X]$, we have
\begin{align*}
\sum_{n = mk} \alpha_m \rho(k, z_1) &= \sum_{\ell = 0}^L (-1)^\ell \sum_{\substack{j_\ell < \dotsb < j_1 \\ Q_u = (1+\delta_A)^{j_u} \\ z_0 \leq Q_\ell < \dotsb < Q_1 < z_1}}\,\, \sum_{\substack{n = emk q_1 \dotsm q_\ell \\ e \mid P(z_0),\, e \leq D \\ Q_u < q_u \leq (1+\delta_A)Q_u}} \alpha_{m} \mu(e) + c_n,
\end{align*}
where $c_n$ are such that, for any $Y \in (X^{99/100}, 4X]$, one has
\[
\sum_{n \sim Y} |c_n|^2 \ll_A \frac{Y}{\log^{2A} X}.
\]
\end{enumerate}
\end{lemma}

\begin{proof}
Let us consider the left-hand side of~\eqref{eq:HarmanDecomposition}. We first use Buchstab's identity and then split the arising prime variable $q$ into short intervals. This gives
\begin{align*}
\sum_{n = mk} \alpha_m \rho(k, z_1) = \sum_{\substack{n = mk}} \alpha_m \rho(k, z_0) - \sum_{\substack{Q =  (1+\delta_A)^j \\ \frac{z_0}{1+\delta_A} \leq Q < z_1}} \sum_{\substack{n = qmk \\ z_0 \leq q < z_1 \\ Q < q \leq Q(1+\delta_A)}} \alpha_m \rho(k, q).
\end{align*}
Hence the formula~\eqref{eq:HarmanDecomposition} holds with $c_n = c_{1, n} + c_{2, n} + c_{3, n}$, where
\begin{align*}
c_{1, n} &:= \sum_{\substack{n = mk}} \alpha_m \rho(k, z_0) - \sum_{\substack{n = emk \\ e \mid P(z_0),\, e \leq D}} \alpha_m \mu(e), \\
c_{2, n} &:= \sum_{\substack{Q = (1+\delta_A)^j \\ z_0 \leq Q < z_1}} \sum_{\substack{n = qmk \\ Q < q \leq Q (1+\delta_A)}} \alpha_m \rho(k, Q) - \sum_{\substack{Q =  (1+\delta_A)^j \\ \frac{z_0}{1+\delta_A} \leq Q < z_1}} \sum_{\substack{n = qmk \\ z_0 \leq q < z_1 \\ Q < q \leq Q(1+\delta_A)}} \alpha_m \rho(k, Q), \\
c_{3, n} &:= \sum_{\substack{Q =  (1+\delta_A)^j \\ \frac{z_0}{1+\delta_A} \leq Q < z_1}} \sum_{\substack{n = qmk \\ z_0 \leq q < z_1 \\ Q < q \leq Q(1+\delta_A)}} \alpha_m \rho(k, Q) - \sum_{\substack{Q =  (1+\delta_A)^j \\ \frac{z_0}{1+\delta_A} \leq Q < z_1}} \sum_{\substack{n = qmk \\ z_0 \leq q < z_1 \\ Q < q \leq Q(1+\delta_A)}} \alpha_m \rho(k, q).
\end{align*}
Note that, for each $j= 1, 2, 3$, we have $|c_{j, n}| \ll d_3(n)$ and hence part (i) follows if we show that, for any $Y \in (X^{99/100}, 4X]$ and any $j = 1, 2, 3$, we have
\begin{align}
\label{eq:cjnupper}
\sum_{n \sim Y} d_3(n) |c_{j, n}| \ll_A \frac{Y}{\log^{3A} X}.
\end{align}

Let us first consider $c_{1,n}$. By~\cite[Lemma 15]{hb-prime} (alternatively see~\cite[Lemma 4.1]{harman-sieves}) we have, for any $\ell \in \mathbb{N}$,
\[
\left|\rho(\ell, z_0) - \sum_{\substack{e \mid (\ell, P(z_0)) \\ e \leq D}} \mu(e)\right| \leq \sum_{\substack{e \mid (\ell, P(z_0)) \\ D < e \leq Dz_0}} 1, 
\]
so that
\[
|c_{1,n}| \ll  \sum_{\substack{n = emk \\ e \mid P(z_0),\, D < e \leq Dz_0}} 1 \leq  \sum_{\substack{n = ek \\ e \mid P(z_0),\, D < e \leq Dz_0}} d(k). 
\]
Hence by~\eqref{eq:divbound}
\begin{align*}
\sum_{n \sim Y} d_3(n) |c_{1, n}| &\ll \sum_{n \sim Y} d_3(n) \sum_{\substack{n = ek \\ e \mid P(z_0) \\ D < e \leq Dz_0}} d(k) \ll \sum_{\substack{e \mid P(z_0) \\ D < e \leq Dz_0}} d_3(e) \sum_{k \sim Y/e} d(k)d_3(k) \\
&\ll Y\log^5 Y \sum_{\substack{e \mid P(z_0) \\ D < e \leq Dz_0}} \frac{d_3(e)}{e}.
\end{align*}
By~\cite[Lemma 4.3]{harman-sieves} (a standard application of Rankin's trick), this is $\ll Y \log^{-3A} Y$.

Let us now turn to $c_{2,n}$. We have
\begin{equation}
\label{eq:c2nbound}
|c_{2,n}| \ll \sum_{\substack{n = qk \\ \frac{z_0}{1+\delta_A} \leq q < z_0(1+\delta_A)}} d(k) + \sum_{\substack{n = qk \\ z_1 \leq q < z_1(1+\delta_A)}} d(k).
\end{equation}
The contribution of the second term in~\eqref{eq:c2nbound} to the left-hand side of~\eqref{eq:cjnupper} with $j=2$ is thus by~\eqref{eq:divbound}
\begin{align*}
&\ll \sum_{n \sim Y} d_3(n) \sum_{\substack{n = qk \\ z_1 \leq q < z_1(1+\delta_A)}} d(k) \ll \sum_{z_1 \leq q < z_1(1+\delta_A)}\, \sum_{k \sim Y/q} d_3(k)d(k) \\
& \ll Y\log^5 Y  \sum_{z_1 \leq q < z_1(1+\delta_A)} \frac{1}{q} \ll \frac{Y}{(\log X)^{3A}}.
\end{align*}
The contribution of the first term in~\eqref{eq:c2nbound} to the left-hand side of~\eqref{eq:cjnupper} with $j=2$ can be similarly shown to be $\ll Y(\log X)^{-3A}$.

Consider finally $c_{3,n}$. If $c_{3,n} \neq 0$, then, for some $Q$ as in the definition of $c_{3, n}$, the integer $n$ has at least two prime factors from $(Q, Q (1+\delta_A)]$. Hence
\[
|c_{3, n}| \ll \sum_{\substack{Q =  (1+\delta_A)^j \\ \frac{z_0}{1+\delta_A} \leq Q < z_1}} \sum_{\substack{n = q_1 q_2 k \\ Q < q_1 \leq q_2 \leq Q (1+\delta_A)}} d(k).
\] 
Hence, using~\eqref{eq:divbound} again,
\begin{align*}
\sum_{n \sim Y} d_3(n) |c_{3, n}| &\ll \sum_{n \sim Y} d_3(n) \sum_{\substack{Q =  (1+\delta_A)^j \\ \frac{z_0}{1+\delta_A} \leq Q < z_1}} \sum_{\substack{n = q_1 q_2 k \\ Q < q_1 \leq q_2 \leq Q (1+\delta_A)}} d(k) \\
&\ll \sum_{\frac{z_0}{1+\delta_A} \leq q_1 < z_1(1+\delta_A)}\,  \sum_{ q_2 \in [q_1, q_1(1+\delta_A)]}\, \sum_{k \sim Y/(q_1 q_2)} d(k) d_3(k) \\
&\ll Y (\log Y)^5  \sum_{\frac{z_0}{1+\delta_A} \leq q_1 < z_1(1+\delta_A)} \, \sum_{ q_2 \in [q_1, q_1(1+\delta_A)]} \frac{1}{q_1 q_2} \ll \frac{Y}{\log^{3A} X}.
\end{align*}
This finishes the proof of part (i). 

To prove part (ii), we claim that for any $J\geq 0$ we have
\begin{align}\label{eq:iterate}\begin{aligned}
\sum_{n = mk} \alpha_m \rho(k, z_1) &= \sum_{\ell = 0}^J (-1)^\ell \sum_{\substack{j_\ell < \dotsb < j_1 \\ Q_u = (1+\delta_A)^{j_u} \\ z_0 \leq Q_\ell < \dotsb < Q_1 < z_1}}\,\, \sum_{\substack{n = emk q_1 \dotsm q_\ell \\ e \mid P(z_0),\, e \leq D \\ Q_u < q_u \leq (1+\delta_A)Q_u}} \alpha_{m} \mu(e)\\
&+(-1)^{J+1}\sum_{\substack{j_{J+1} < \dotsb < j_1 \\ Q_u = (1+\delta_A)^{j_u} \\ z_0 \leq Q_{J+1} < \dotsb < Q_1 < z_1}}\,\, \sum_{\substack{n = mk q_1 \dotsm q_{J+1} \\ Q_u < q_u \leq (1+\delta_A)Q_u}} \alpha_{m} \rho(k,Q_{J+1})+ c_{n,J}
\end{aligned}
\end{align}
with $c_{n,J}$ satisfying for any $Y \in (X^{99/100}, 4X]$ the bound
\begin{align}\label{eq:l2}
\sum_{n\sim Y}|c_{n,J}|^2\ll_A\frac{Y}{(\log X)^{3A}}\left\lceil \frac{\log(4X)}{\log z_0}\right\rceil^{\sum_{0\leq j\leq J} 2j}.
\end{align}
For $J=0$, we have this by part (i). Supposing then that we have this for some $J$, the case $J+1$ follows by applying part (i) with $z_1=Q_{J+1}$ to the new sequence
\begin{align*}
 \alpha_{r}'= \sum_{\substack{j_{J+1} < \dotsb < j_1 \\ Q_u = (1+\delta_A)^{j_u} \\ z_0 \leq Q_{J+1} < \dotsb < Q_1 < z_1}}\,\,\sum_{r=mq_1\cdots q_{J+1}}\alpha_m,
\end{align*}
which is bounded by $\lceil (\log 4X)/(\log z_0)\rceil^{J+1}$ times the maximum of $\alpha_m$.
Hence,~\eqref{eq:iterate} holds with the bound~\eqref{eq:l2}.

Part (ii) follows from~\eqref{eq:iterate} and~\eqref{eq:l2} with $J=L$, since $((\log 4X)/(\log z_0)+1)^{2L^2}=(\log X)^{o(1)}$ and since for $J=L$ the last sum in~\eqref{eq:iterate} is empty, as a number $n\leq 4X$ cannot have more than $L$ prime factors that are $\geq z_0$.
\end{proof}

\begin{proof}[Proof of Proposition~\ref{prop_rho}]
Recall from~\eqref{eq:rho-def} that
\begin{align*}
\rho^-(n) = \rho(n, z) - \sum_{\substack{n = qm \\ z \leq q < 2X^{1/2}}} \rho(m, z) +  \sum_{\substack{n = q_1 q_2 m \\ z \leq q_2 < q_1 < X^{1/4-2\varepsilon} \\ q_1 q_2^4 < X^{1-2\varepsilon}}} \rho(m, z) - \sum_{\substack{n = q_1 q_2 q_3 m \\ z \leq q_3 < q_2 < q_1 < X^{1/4-2\varepsilon} \\ q_1 q_2^4 < X^{1-2\varepsilon}}} \rho(m, z).
\end{align*}

Let us concentrate on the third term, the other terms being treated similarly (except the fourth term leads to type I/II sums instead of type I sums). We first split the variables $q_1$ and $q_2$ into $\delta_A$-adic ranges and write 
\begin{align}
\nonumber
&1_{n \in (Y/2, 4Y]} \sum_{\substack{n = q_1 q_2 m \\ z \leq q_2 < q_1 < X^{1/4-2\varepsilon} \\ q_1 q_2^4 < X^{1-2\varepsilon}}} \rho(m, z) \\
\label{eq:afteradic}
&=1_{n \in (Y/2, 4Y]} \sum_{\substack{Q_1 = (1+\delta_A)^{j_1}, Q_2 = (1+\delta_A)^{j_2} \\ z \leq Q_2 < Q_1 < X^{1/4-2\varepsilon} \\ Q_1 Q_2^4 < X^{1-2\varepsilon}}} \sum_{\substack{n = q_1 q_2 m \\ Q_j < q_j \leq (1+\delta_A)Q_j}} \rho(m, z) + d_{1, n}.
\end{align}
We can show that the mean square of $d_{1, n}$ is small by arguing as when we treated $c_{3, n}$ in the proof of Lemma~\ref{le:HarmanSieve}, so $d_{1,n}$ can be included in $c_n$ in the statement of Proposition~\ref{prop_rho}.

Applying Lemma~\ref{le:HarmanSieve}(ii) (taking $\alpha_m = 1_{m = q_1q_2 \in \mathcal{P}_2}$, where $\mathcal{P}_2$ is defined by the summation conditions above) and splitting also the arising variables $e$ and $k$ from Lemma~\ref{le:HarmanSieve} into $\delta_A$-adic intervals, we see that the main term in~\eqref{eq:afteradic} is a linear combination of an acceptable error and
\[
\ll L \left(\frac{\log(z/z_0)}{\log(1+\delta_A)}\right)^{L} (\log X)^3 \log D \ll \exp((\log \log X)^{5})
\]
terms of the form
\begin{align*}
&1_{n \in (Y/2, 4Y]} \sum_{\substack{n = eq_1q_2k q'_1 \dotsm q'_\ell \\ e \mid P(z_0), E < e \leq E(1+\delta_A) \\ Q_j < q_j \leq (1+\delta_A)Q_j \\ Q'_u < q'_u \leq (1+\delta_A)Q'_u \\ K < k \leq K(1+\delta_A)}} \mu(e),
\end{align*}
where $\ell \leq L$,
\begin{align*}
z \leq Q_2 < Q_1 < X^{1/4-2\varepsilon}, \quad Q_1 Q_2^4 < X^{1-2\varepsilon}, \quad z_0 \leq Q'_\ell < \dotsb < Q'_1 < z, \text{and } E \leq D.
\end{align*}

Parts of the linear combination where $EQ_1Q_2 K Q'_1 \dotsm Q'_\ell \not \in (Y/2, 4Y]$ make an acceptable contribution arguing as with $c_{2,n}$ in the proof of Lemma~\ref{le:HarmanSieve}. Once we have imposed this condition, a similar argument allows us to dispose of the factor $1_{n \in (Y/2, 4Y]}$.

When $Q'_1 \dotsm Q'_\ell < X^{\varepsilon/2}$, we have $E Q_1 Q_2 Q'_1 \dotsm Q'_\ell \leq X^{1/2+\varepsilon}$, and hence we have a type I sum. On the other hand, when $Q_1' \dotsm Q_\ell' \geq X^{\varepsilon/2}$, one can find $I \subseteq \{1, \dotsc, \ell\}$ such that $X^{\varepsilon/2}\leq \prod_{i\in I}Q_i' \leq z$ and hence we have a type II sum.
\end{proof}

\section{Mean value theorems for Dirichlet polynomials}
As usual, the variance estimate in Theorem~\ref{thm_minorant}(iii) is reduced to a mean square estimate for corresponding Dirichlet polynomials. For the following lemma, see e.g.~\cite[Lemma 1]{tera-primes}.

\begin{lemma}[Reduction to Dirichlet polynomials]
\label{lem:DirPolRed} 
Let $X\geq 3$, $T_0=X^{1/1000}$ and let $\varepsilon>0$ be small enough but fixed. Let $c=2.1, h=(\log X)^c, a \in [c-1-1/10000, c-1]$, and $P_1 = (\log X)^a$. Define
\begin{align*}
P_1(s):=\sum_{\substack{p_1 \sim P_1}} \frac{1}{p_1^{s}},\quad P(s):=\sum_{X/(2P_1)<  n\leq  4X/P_1}\frac{\rho^-(n)}{n^{s}}.
\end{align*}
Suppose that, for any $T \geq X/h$,
\begin{align*}
\int_{T_0}^{T}|P_1(1+it)|^2|P(1+it)|^2\d t\ll \frac{T}{X/h} \cdot \frac{1}{(\log X)^{2+\varepsilon}}.    
\end{align*}
Then~\eqref{eq:minMSbound} holds. 
\end{lemma}

In this section we collect some mean value theorems for Dirichlet polynomials that we shall need.
Let us start with the standard mean value theorem (see~\cite[Theorem 9.1]{iw-kow}).
\begin{lemma}[Mean value theorem for Dirichlet polynomials]
\label{le:contMVT}
Let $N, T \geq 1$ and $A(s) = \sum_{n \leq N} a_n n^{-s}$. Then
\[
\int_{-T}^T |A(it)|^2 \d t = (2T+O(N)) \sum_{n \leq N} |a_n|^2.
\]
\end{lemma}
We will also need the following variant, which works better when $a_n$ has sparse support.
\begin{lemma}[Improved mean value theorem]
\label{le:contMVT2}
Let $N, T \geq 1$ and $A(s) = \sum_{n \leq N} a_n n^{-s}$. Then
\[
\begin{split}
\int_{-T}^T |A(it)|^2 \d t &\ll T \sum_{n \leq N} |a_n|^2 + T \sum_{0 < |k| \leq N/T} \sum_{n \leq N} |a_n| |a_{n+k}|.
\end{split}
\]
\end{lemma}
\begin{proof}
This follows from~\cite[Lemma 7.1]{iw-kow} taking $Y = 10T$ and $x_m = \frac{1}{2\pi}\log m$ there.
\end{proof}

The following sparse mean value estimate of Heath-Brown~\cite{hb-smooth} plays an important role in our arguments.  

\begin{lemma}[Heath-Brown's sparse mean value estimate]\label{le_HB_MVT} Let $T\geq M\geq 1$,  let $\mathcal{M}\subset [M,T]$ be a set of integers, and let $N\geq 2$. Let 
\begin{align*}
M(s)=\sum_{m\in \mathcal{M}}\varepsilon_mm^{-s} \quad \textnormal{ with } \quad |\varepsilon_m|\leq 1    
\end{align*}
and
\begin{align*}
A(s)=\sum_{n \sim N}a_n n^{-s}
\end{align*}
for some $a_n \in \mathbb{R}$. Then we have, for any $\eta > 0$,
\begin{align}\label{eq_HB}
\int_{-T}^{T}|M(1+it)|^2|A(1+it)|^2\d t\ll_{\eta} \left( \left(\frac{|\mathcal{M}|}{M}\right)^2+(NT)^\eta\left(\frac{|\mathcal{M}|T}{M^2N}+ \frac{|\mathcal{M}|^{7/4}T^{3/4}}{M^2N}\right)\right)\max_{n}|a_n|^2.    
\end{align}
Moreover, if $N\geq T^{2/3}$ or $|\mathcal{M}|\leq T^{1/3}$, the third term on the right-hand side of~\eqref{eq_HB} can be deleted.
\end{lemma}

\begin{proof}
This follows quickly from~\cite[Theorem 4]{hb-smooth}. Firstly, by symmetry, it suffices to consider the integral over $[0,T]$. Secondly, by writing $A(s)=\widetilde{A}(s)\max_{n}|a_n|$ with $\widetilde{A}(s)$ of the form $\sum_{n\sim N}\widetilde{a_n}n^{-s}$ with $|\widetilde{a_n}|\leq 1$, it suffices to consider the case $|a_n|\leq 1$. Next, by writing 
\begin{align*}
a_n =a_n^{+}-a_n^{-}
\end{align*}
with $a^{\pm}_n\in [0,1]$
and applying the triangle inequality, it suffices to consider the case $a_n\in [0,1]$. Lastly, we can write
\begin{align*}
M(1+it)=\frac{M_1(it)}{M}\quad \text{with} \quad M_1(it)=\sum_{m\in \mathcal{M}}\frac{\varepsilon_mM}{m}m^{-it}    
\end{align*}
and
\begin{align*}
A(1+it)=\frac{A_1(it)}{N}\quad \text{with} \quad A_1(it)=\sum_{n \sim N}\frac{a_nN}{n}n^{-it}    
\end{align*}
to reduce matters to an integral over the $0$-line. Now the claim follows from~\cite[Theorem 4(iii)]{hb-smooth} (with $\eta/2$ in place of $\eta$). 
\end{proof}

To obtain Type I and Type I/II information we use twisted moment estimates due to Watt~\cite{watt} and Deshouillers--Iwaniec~\cite{DI}.

\begin{lemma}
\label{le:WattDI} Let $A, N, N', T\geq 1$ with $N<N'\leq 2N$. Let 
\begin{align*}
N(s)=\sum_{N < n \leq N'}n^{-s},\quad A(s)=\sum_{m\sim A}a_mm^{-s}    
\end{align*}
with $a_n$ complex numbers, and let $\varepsilon > 0$. 
\begin{enumerate}[(i)]
\item (Watt's theorem) We have
\begin{align*}
\int_{T/2}^{T}|N(1+it)|^4|A(1+it)|^2\d t\ll T^{\varepsilon} \left(\frac{T+A^2T^{1/2}}{N^2A}+\frac{T+A}{T^4A}\right)\max_{m}|a_m|^2
\end{align*}
\item (Deshouillers--Iwaniec theorem) We have
\begin{align*}
\int_{T/2}^{T}|N(1+it)|^4|A(1+it)|^2 \d t\ll  T^{\varepsilon}\left(\frac{T+A^2T^{1/2}+A^{5/4}T^{3/4}}{N^2A}+\frac{T+A}{T^4A}\right)\frac{1}{A}\sum_{m\sim A}|a_m|^2.
\end{align*}
\end{enumerate}
\end{lemma}

\begin{remark} Although the $A$ dependence is weaker in Lemma~\ref{le:WattDI}(ii) than in Lemma~\ref{le:WattDI}(i), the fact that Lemma~\ref{le:WattDI}(ii) involves the $\ell^2$ norm of the coefficient sequence rather than the maximum makes it more suited for our type I estimates in Section~\ref{subsec:typeI}, where we essentially end up taking $A(s)=P_1(s)^k$ with $P_1(s)=\sum_{p\sim P_1}p^{-s}$ and $P_1^k\approx T^{1/10}$. Indeed, in this situation, taking the maximum of the coefficient sequence would lead to a loss in the estimate. 
\end{remark}

\begin{proof}[Proof of Lemma~\ref{le:WattDI}]
Suppose first that $N\leq T$. Then, by the approximate functional equation, we can further reduce to $N\ll T^{1/2}$ (cf.~\cite[formula (5.6.12)]{harman-sieves}). Parts (i) and (ii) then follow from the works of Watt~\cite{watt} and Deshouillers--Iwaniec~\cite{DI}, respectively (one can apply partial summation to change the line of integration and then use~\cite[(4.7)]{watt} and~\cite[(14)]{DI}).

Suppose then that $N>T$. For $N>T\geq |t|$, we have the bound
\begin{align}\label{eq:nit}
\sum_{N< n\leq N'}n^{-1-it}\ll \frac{1}{|t|};
\end{align}
this follows e.g. from~\cite[Corollary 8.11]{iw-kow}. The claimed estimates then follow from~\eqref{eq:nit} and the mean value theorem (Lemma~\ref{le:contMVT}) applied to $A(1+it)$.
\end{proof}

\section{Large values of Dirichlet polynomials}
\begin{lemma}["Density hypothesis" for Dirichlet polynomials] \label{le_laregvalues} Let $\varepsilon>0$ be small but fixed, and let $T \geq N \geq 2$. Let $N(s)=\sum_{n \sim N}a_nn^{-s}$ with $a_n$ divisor-bounded. Assume that $T^{1-\varepsilon/10} \geq N \geq T^{9/11-10\varepsilon}$ and 
\[
5\varepsilon \leq \sigma \leq 49/206-5\varepsilon.
\]
Then we have
\[
|\{t \in [-T, T]:|N(1+it)|>N^{-\sigma}\}| \ll T^{2\sigma-\varepsilon^2/5}.
\]
\end{lemma}

\begin{proof}
Write 
\[
\mathcal{T} = \{t \in [-T, T]: |N(1+it)|>N^{-\sigma}\}.
\]
We apply Jutila's large values estimate (see~\cite[Theorem 9.10]{iw-kow}) with $k=7,$
\begin{equation*}
G = \sum_{n \sim N}\frac{|a_n|^2}{n^2} \ll \frac{(\log N)^{O(1)}}{N} \quad \text{and} \quad V = N^{-\sigma}. 
\end{equation*}
This gives
\begin{align*}
|\mathcal{T}| &\ll (NT)^{\varepsilon^2/500} \left(\frac{GN}{V^2} + \left(\frac{GN}{V^2}\right)^{-1/k} \frac{G^3 NT}{V^6} + \left(\frac{GN}{V^2}\right)^{4k} \frac{T}{N^{2k}}\right) \\
&\ll T^{\varepsilon^2/200} \left(N^{2\sigma} + T N^{(6-2/7)\sigma-2} +T N^{(8\sigma -2) \cdot 7}  \right) \ll T^{2\sigma-\varepsilon^2/5},
\end{align*}
as desired.
\end{proof}
Lemma~\ref{le_laregvalues} allows us to handle large values of Dirichlet polynomials in our type II sums.

\begin{proposition}[Type II estimate]\label{prop_typeII}
Let $\varepsilon>0$ be small enough but fixed. Let $T\geq 3$, and let
\[
M_1(s) = \sum_{\substack{m_1 \sim M_1}} \frac{\alpha_{m_1}}{m_1^s} \quad \text{and} \quad M_2(s) = \sum_{\substack{m_2 \sim M_2}} \frac{\beta_{m_2}}{m_2^s}  
\]
with $M_1M_2=T(\log T)^{O(1)}$ and with $(\alpha_{m_1})$ and $(\beta_{m_2})$ divisor-bounded. Suppose that 
\[
T^{\varepsilon/5} \leq M_1 \leq T^{2/11+\varepsilon}
\]
and
\[
\sigma = 49/206-10\varepsilon.
\]
Let $\mathcal{U} \subseteq [0, T]$ be a measurable set such that, for each $t \in \mathcal{U}$, one has either $|M_1(1+it)| \geq M_1^{-\sigma}$ or $|M_2(1+it)| \geq M_2^{-\sigma}$. 
Then
\begin{align}\label{eq_m1m2}
\int_{\mathcal{U}} |M_1(1+it)|^2|M_2(1+it)|^2 \d t \ll T^{-\varepsilon^2/10} + (\log T)^{O(1)} \sup_{t \in \mathcal{U}} |M_1(1+it)|^2.
\end{align}
\end{proposition}

\begin{remark}\label{rem:typeII} The key aspect in Proposition~\ref{prop_typeII} is the value of $\sigma$, which we want to maximize, as the value of $\sigma$ eventually plays an important role in determining our exponent $c$ in Theorem~\ref{thm_minorant} in Subsection~\ref{subsec:typeII}. 

In~\cite{tera-primes}, integrals of the type~\eqref{eq_m1m2} were estimated by using a pointwise bound on $M_1(1+it)$ and the Hal\'asz--Montgomery inequality on the sparse mean square of $M_2(1+it)$, whereas in the proof of Proposition~\ref{prop_typeII} we obtain stronger estimates by first splitting the integral into pieces according to the sizes of $M_1(1+it)$, $M_2(1+it)$ and then applying Jutila's large values estimate. 
\end{remark}

\begin{proof}
It suffices to show that, for any one-spaced subset $\mathcal{T} \subseteq \mathcal{U}$, we have
\begin{align*}
\sum_{t \in \mathcal{T}} |M_1(1+it)|^2|M_2(1+it)|^2 \ll T^{-\varepsilon^2/10} + (\log T)^{O(1)} \sup_{t \in \mathcal{U}} |M_1(1+it)|^2.   
\end{align*}
We partition $\mathcal{T} = \mathcal{T}_1 \cup \mathcal{T}_2 \cup \mathcal{T}_3$, where
\begin{align*}
\mathcal{T}_1 &:= \{t \in \mathcal{T} \colon |M_1(1+it)| \geq M_1^{-10\varepsilon} \text{ or } |M_2(1+it)| \geq M_2^{-10\varepsilon} \}, \\
\mathcal{T}_2 &:= \{t \in \mathcal{T} \colon T^{-1} \leq |M_1(1+it)| < M_1^{-10\varepsilon} \text{ and } T^{-1} \leq |M_2(1+it)| < M_2^{-10\varepsilon} \}, \\
\mathcal{T}_3 &:= \{t \in \mathcal{T} \colon |M_1(1+it)| < T^{-1} \text{ or } |M_2(1+it)| < T^{-1} \}.
\end{align*}
Trivially
\[
\sum_{t \in \mathcal{T}_3} |M_1(1+it)|^2|M_2(1+it)|^2 \ll T \cdot \frac{1}{T^2} \cdot (\log T)^{O(1)} \ll \frac{1}{T^{1/2}}. 
\]
Let us turn to $\mathcal{T}_1$. Let $\ell_j:=\lceil (\log T)/(\log M_j)\rceil$ for $j=1,2$.
Now
\[
|\mathcal{T}_1|\ll M_1^{20\varepsilon\ell_1}\sum_{t\in \mathcal{T}}|M_1(1+it)|^{2\ell_1} + M_2^{20\varepsilon\ell_2}\sum_{t\in \mathcal{T}}|M_2(1+it)|^{2\ell_2}.
\]
Note that $\ell_j\ll 1$ and thus the coefficients of $M_j(s)^{\ell_j}$ are divisor-bounded. Hence by the discrete mean value theorem~\cite[Theorem 9.4]{iw-kow}, for $j = 1, 2$,
\begin{align*}
M_j^{20\varepsilon\ell_j}\sum_{t\in \mathcal{T}}|M_j(1+it)|^{2\ell_j}\ll 
T^{40\varepsilon}(T+(2M_j)^{\ell_j}) \frac{(\log T)^{O(1)}}{M_j^{\ell_j}} \ll T^{50\varepsilon}.
\end{align*}
Hence $|\mathcal{T}_1| \ll T^{50\varepsilon}$. 

Using the pointwise bound for $M_1(1+it)$ and the Hal\'asz--Montgomery inequality (see~\cite[Theorem 9.6]{iw-kow}) for $M_2(1+it)$, we obtain that
\begin{align*}
\sum_{t\in \mathcal{T}_1}|M_1(1+it)|^2|M_2(1+it)|^2 &\ll \sup_{t \in \mathcal{U}} |M_1(1+it)|^2 \left(1+\frac{|\mathcal{T}_{1}|T^{1/2}}{M_2}\right) (\log T)^{O(1)} \\
&\ll (\log T)^{O(1)} \sup_{t \in \mathcal{U}} |M_1(1+it)|^2. 
\end{align*}

Hence we can concentrate on $\mathcal{T}_2$. We partition the set $\mathcal{T}_2$ into $\ll (\log T)^{2}$ sets of the form 
\begin{align*}
\mathcal{T}_{2, \sigma_1, \sigma_2}=\{t\in \mathcal{T}_2:\,\, |M_1(1+it)|\in (M_1^{-\sigma_1}, 2M_1^{-\sigma_1}],\,\, |M_2(1+it)|\in (M_2^{-\sigma_2}, 2M_2^{-\sigma_2}]\}.    
\end{align*}
Note that the set is non-empty only if $\sigma_1, \sigma_2 \geq 10\varepsilon$ and $\min\{\sigma_1, \sigma_2\} \leq \sigma+1/\log M_1$. If $\sigma_1 \leq \sigma_2$, we choose an integer $1\leq \ell\ll 1$ such that $M_1^\ell \in [T^{5/6-\varepsilon}, T^{1-\varepsilon/2}]$ and apply Lemma~\ref{le_laregvalues}  to $M_1(1+it)^\ell$ (note that the coefficients of $M_1(s)^\ell$ are divisor-bounded). If $\sigma_2 < \sigma_1$ we apply Lemma~\ref{le_laregvalues} to $M_2(1+it)$ (which has length $\in [T^{9/11-\varepsilon}(\log T)^{-O(1)},T^{1-\varepsilon/5}(\log T)^{O(1)}]$). We obtain
\[
|\mathcal{T}_{2,\sigma_1,\sigma_2}| \ll T^{2\min\{\sigma_1, \sigma_2\}-\varepsilon^2/5},
\]
and consequently 
\[
\sum_{t\in \mathcal{T}_2}|M_1(1+it)|^2|M_2(1+it)|^2 \ll (\log T)^2 M_1^{-2\sigma_1} M_2^{-2\sigma_2} T^{2\min\{\sigma_1, \sigma_2\}-\varepsilon^2/5} \ll T^{-\varepsilon^2/10}.
\]
This completes the proof.
\end{proof}

\section{Proof of Theorem~\ref{thm_minorant}} \label{sec:proof}

Let $c=2.1$, $a \in [c-1-1/10000, c-1]$, $h=(\log X)^c$ and $P_1=(\log X)^{a}$ as in Theorem~\ref{thm_minorant}, and let $\rho^-$ be as in~\eqref{eq:rho-def} (with $z=X^{2/11}$). Write
\begin{align}\label{eq:pdef}
P_1(s)=\sum_{p_1\sim P_1}\frac{1}{p_1^s} \quad \text{and} \quad P(s)=\sum_{X/(2P_1) < n\leq  4X/P_1}\frac{\rho^-(n)}{n^{s}}.
\end{align}
Let also
\begin{align*}
T_0=X^{1/1000} \quad \text{and} \quad T_1 = X.
\end{align*}

By Lemma~\ref{lem:DirPolRed} it suffices to show that there exists $\varepsilon > 0$ such that, for any $T \geq X/h$,
\begin{align}\label{eq:dirpol}
\int_{T_0}^{T}|P_1(1+it)|^2|P(1+it)|^2\d t\ll \frac{T}{X/h} \cdot \frac{1}{(\log X)^{2+\varepsilon/10}}.  
\end{align}
If $T \geq T_1$, the mean value theorem (Lemma~\ref{le:contMVT}) and~\eqref{eq:rhobound} imply that
\begin{align*}
\int_{T_0}^{T}|P_1(1+it)|^2|P(1+it)|^2\d t\ll (T + X) \sum_{n \sim X} \frac{1}{n^2} \ll \frac{T}{X} \cdot \frac{h}{(\log X)^{2+\varepsilon/10}}
\end{align*}
since $h = (\log X)^c > (\log X)^{2+\varepsilon/10}$.

Hence we can assume that $T \in [X/h, T_1]$. We separate into two cases according to the size of $P_1(1+it)$. Let $[T_0, T_1] = \mathcal{T} \cup \mathcal{U}$, where
\begin{align*}
\mathcal{U}:=\{t\in [T_0,T_1]\colon |P_1(1+it)|\geq P_1^{-\varepsilon/10}\}.
\end{align*}
Recall from~\eqref{eq:rhobound} that $|\rho^-(n)| \ll \rho(n, z)$, so that by the improved mean value theorem (Lemma~\ref{le:contMVT2}) and a simple sieve upper bound (similar to e.g.~\cite[Theorem 6.7]{iw-kow}) we have for any $T \in [X/h, T_1]$,
\begin{align*}
\int_{-T}^T |P(1+it)|^2 \d t &\ll T \sum_{X/(2P_1)\leq n\leq  4X/P_1} \frac{1_{(n, P(z)) = 1}}{n^2}\\
&+ T\sum_{0<|k| \leq 4X/(P_1T)}\, \sum_{X/(2P_1)\leq n\leq  4X/P_1}  \frac{1_{(n(n+k), P(z)) = 1}}{n(n+k)} \\
&\ll \left(\frac{TP_1}{X \log X} + \frac{1}{(\log X)^2}\right)
\end{align*}
Using also the pointwise estimate $|P(1+it)| \leq P_1^{-\varepsilon/10}$ for $t \in \mathcal{T}$ and~\eqref{eq:aP1def}, we obtain, for any $T \in [X/h, T_1]$,
\begin{align*}
&\int_{\mathcal{T} \cap [T_0, T]}|P_1(1+it)|^2|P(1+it)|^2 \d t \ll \frac{1}{P_1^{\varepsilon/5}}  \left(\frac{TP_1}{X \log X} + \frac{1}{(\log X)^2}\right) \\
& \ll \frac{1}{(\log X)^{a \varepsilon/5}}  \left(\frac{Th/\log X}{X \log X} + \frac{1}{(\log X)^2}\right) \ll \frac{T}{X/h} \cdot \frac{1}{(\log X)^{2+\varepsilon/5}}.
\end{align*}

Hence, it suffices to show that
\begin{align*}
\int_{\mathcal{U}} |P(1+it)|^2 \d t \ll (\log X)^{-10}.    
\end{align*}

Let $A$ be sufficiently large. Recall the decomposition of $\rho^-(n) 1_{n \in (Y/2, 4Y]}$ from Proposition~\ref{prop_rho}. We pick $Y = X/P_1$ and let $c_n$ and $\mathcal{F}$ be as in Proposition~\ref{prop_rho}. By the mean value theorem (Lemma~\ref{le:contMVT}), and~\eqref{eq:cnbound}
\begin{align*}
\int_{\mathcal{U}} \left|\sum_{n} \frac{c_n}{n^{1+it}}\right|^2 \d t \ll \left(T_1+\frac{X}{P_1}\right) \left(\frac{P_1}{X}\right)^2 \sum_{n \asymp X/P_1} |c_n|^2 \ll_A (\log X)^{-A}.    
\end{align*}
Hence, recalling that $|\mathcal{F}| \ll \exp((\log \log X)^5)$, it suffices to show that, for any $f \in \mathcal{F}$, we have
\begin{align} 
\label{eq:Guintclaim}
\int_{\mathcal{U}} |F(1+it)|^2 \d t \ll \exp(-(\log \log X)^6),
\end{align}
where
\[
F(s) = \sum_{n} \frac{f(n)}{n^s}.
\]
We split into three cases as in Proposition~\ref{prop_rho}. In all three cases we utilize, similarly to~\cite{matomaki-radziwill}, the fact that $|P_1(1+it)| \geq P_1^{-\varepsilon/10}$ for every $t \in \mathcal{U}$ through inserting a factor $|P_1(1+it)|^{2k} P_1^{2k\varepsilon}$ for an appropriate $k$ to the left-hand side of~\eqref{eq:Guintclaim}.

\subsection{Type I case}\label{subsec:typeI}
Now $F(s) = M_1(s) M_2(s)$ with
\[
M_1(s) = \sum_{\substack{m_1 \sim M_1}} \frac{\alpha_{m_1}}{m_1^s} \quad \text{and} \quad M_2(s) = \sum_{\substack{M_2 < m_2 \leq (1+\delta_A) M_2}} \frac{1}{m_2^s}, 
\]
where $M_1 \leq X^{1/2+\varepsilon}$, $M_1 M_2 \in (X/(2P_1), 4X/P_1]$, and $\delta_A$ is given by~\eqref{eq:delta}. It suffices to show that, for any $T \in [T_0, T_1]$, we have
\[
\int_{\mathcal{U} \cap [T, 2T]} |M_1(1+it)|^2 |M_2(1+it)|^2 \d t \ll \exp(-(\log \log X)^6)(\log X)^{-1}.
\]

Let
\[
k = \left\lceil \frac{\log T^{1/10}}{\log P_1} \right\rceil,
\] 
so that $M := P_1^k \in [T^{1/10},P_1 T^{1/10}]$. 
Let
\begin{align*}
M(s)=P_1(s)^{k}=\sum_{P_1^{k} < m \leq (2P_1)^k} \frac{b_{m}}{m^{s}},    
\end{align*}
say. Note that since $P_1=(\log X)^{a}$ we have
\begin{align*}
|b_m| \leq k! \leq \exp(k \log k) = \exp\left(\frac{\log M}{\log P_1} \log \frac{\log M}{\log P_1}\right) \leq M^{1/a},
\end{align*}
so that
\begin{align}
\label{eq:bMS}
\sum_m |b_{m}|^2 \ll M^{1/a} \sum_m |b_m| \ll M^{1/a}P_1^k = M^{1+1/a}.    
\end{align}
By definition we have, for every $t\in \mathcal{U}$,
\begin{align*}
|M(1+it)|\geq M^{-\varepsilon/10} \iff 1 \leq M^{\varepsilon/10} |M(1+it)|.
\end{align*}

Using this and the Cauchy--Schwarz inequality we obtain
\begin{align*}
& \int_{\mathcal{U} \cap [T, 2T]} |M_1(1+it)|^2|M_2(1+it)|^2 \d t\\
\ll& M^{\varepsilon/10}\left(\int_{\mathcal{U} \cap [T, 2T]} |M_2(1+it)|^4|M(1+it)|^2 \d t\right)^{1/2}\left(\int_{\mathcal{U} \cap [T, 2T]} |M_1(1+it)|^4 \d t\right)^{1/2}.
\end{align*}
We apply the Deshouillers--Iwaniec mean value bound (Lemma~\ref{le:WattDI}(ii)) together with~\eqref{eq:bMS} to the first term and the mean value theorem (Lemma~\ref{le:contMVT}) to the second term, obtaining that the above is
\[
\ll M^{\varepsilon/10} T^\varepsilon\left(\frac{T + M^2 T^{1/2} + M^{5/4} T^{3/4}}{M_2^2 M} + \frac{T+M}{T^4 M} \right)^{1/2} M^{1/(2a)} \left(\frac{T+M_1^2}{M_1^2}\right)^{1/2}.
\]
Let us first note that since $M \leq T^{1/5}$, the contribution corresponding to the term $(T+M)/(T^4 M)$ is
\[
\ll M^{\varepsilon/10} T^\varepsilon \left(\frac{(T+M)(T+M_1^2)}{T^4MM_1^2}\right)^{1/2}M^{1/(2a)} \ll \frac{M^{1/(2a)+\varepsilon/10}}{T^{1-\varepsilon}}  \ll \frac{1}{T_0^{1/2}}\ll \frac{1}{X^\varepsilon}.
\]
The remaining terms are maximal when $T$ is maximal, i.e. $T = T_1= X$. In this case $T + M^2 T^{1/2} + M^{5/4} T^{3/4} \ll T$ and $M_1^2\leq X^{1+2\varepsilon}\leq T^{1+2\varepsilon}$, and the bound we obtain is
\[
\ll M^{\varepsilon/10} T^{2\varepsilon} \left(\frac{T^2}{M_1^2 M_2^2 M^{1-1/a}}\right)^{1/2} \ll \frac{1}{X^{\varepsilon}}.
\]
\begin{remark}
One could slightly loosen the condition $M_1 \leq X^{1/2+\varepsilon}$. The above argument with $k$ such that $P_1^k \approx T^{1/5-\varepsilon}$ would allow one to handle type I sums for $M_1 \leq X^{1/2+(1-1/a)/10-10\varepsilon}$.
\end{remark}

\subsection{Type I/II case}\label{subsec:typeI/II}
Now $F(s) = M_1(s) M_2(s) M_3(s)$ with
\begin{align*}
M_1(s) = \sum_{\substack{m_1 \sim M_1}} \frac{\alpha_{m_1}}{m_1^s}, \quad M_2(s) =  \sum_{\substack{m_2 \sim M_2}} \frac{\alpha_{m_2}}{m_2^s}, \quad M_3(s) =  \sum_{\substack{M_3 < m_3 \leq (1+\delta_A) M_3}} \frac{1}{m_3^s}
\end{align*}
and
\begin{align}\label{eq_cond3}
M_1^2M_2\leq X^{1-\varepsilon},\quad M_2\leq X^{1/4-\varepsilon}, \quad \text{and} \quad M_1 M_2 M_3 \in (X/(2P_1), 4X/P_1].
\end{align}
Similarly to Section~\ref{subsec:typeI}, it suffices to show that, for any $T \in [T_0, T_1]$, we have
\[
\int_{\mathcal{U} \cap [T, 2T]} |M_1(1+it)|^2 |M_2(1+it)|^2 |M_3(1+it)|^2 \d t \ll \exp(-(\log \log X)^6)(\log X)^{-1}
\]
We argue similarly to the type I case in Section~\ref{subsec:typeI}, but this time taking $M(s) := P_1(s)^k$ with
\[
k = \left\lceil \frac{\log T^{\varepsilon/2}}{\log P_1} \right\rceil,
\] 
so that $M := P_1^k \in [T^{\varepsilon/2},P_1 T^{\varepsilon/2}]$. By the Cauchy--Schwarz inequality, we have
\begin{align*}
& \int_{\mathcal{U} \cap [T, 2T]} |M_1(1+it)|^2|M_2(1+it)|^2|M_3(1+it)|^2 \d t\\
\ll& M^{\varepsilon/10}\left(\int_{\mathcal{U} \cap [T, 2T]} |M_3(1+it)|^4|M_2(1+it)|^2 \d t\right)^{1/2}\left(\int_{\mathcal{U} \cap [T, 2T]} |M_1(1+it)|^4|M_2M(1+it)|^2 \d t\right)^{1/2}.
\end{align*}
By Watt's bound (Lemma~\ref{le:WattDI}(i)), the mean value theorem (Lemma~\ref{le:contMVT}) and~\eqref{eq:bMS}, this is
\[
\ll M^{\varepsilon/10} T^{\varepsilon/100} \left(\frac{T+M_2^2 T^{1/2}}{M_3^2M_2}+\frac{T+M_2}{T^4 M_2}\right)^{1/2} \left(\frac{T + M_1^2 M_2 M}{M_1^2 M_2 M} \right)^{1/2}M^{1/(2a)}.
\]

Let us first note that since $M \leq T^\varepsilon$, the contribution corresponding to the term $(T+M_2)/(T^4 M_2)$ is
\[
\ll M^{\varepsilon/10} T^{\varepsilon/100} \left( \frac{(T+M_2)(T+M_1^2M_2M)}{T^4M_1^2M_2^2M}\right)^{1/2}M^{1/(2a)} \ll \frac{M^{1/(2a)+\varepsilon/10}}{T^{1-\varepsilon/100}} \ll \frac{1}{T_0^{1/2}}\ll \frac{1}{X^\varepsilon}.
\]
The remaining terms are maximal when $T$ is maximal, i.e. $T = T_1 = X$. In this case $T + M_2^2 T^{1/2} \ll T$ and $M_1^2 M_2 M \ll T$, and the bound we obtain is
\[
\ll M^{\varepsilon/10} T^{\varepsilon/100} \left(\frac{T^2}{M_1^2 M_2^2 M_3^2 M^{1-1/a}}\right)^{1/2}.
\]
By~\eqref{eq_cond3} this is $\ll X^{-\varepsilon/100}$.

\subsection{Type II case} \label{subsec:typeII}

Now $F(s) = M_1(s) M_2(s)$ with
\begin{align*}
M_1(s) = \sum_{M_1 < m_1 \leq (1+\delta_A)^R M_1} \frac{\alpha_{m_1}}{m_1^s}, \quad M_2(s) = \sum_{m_2\sim M_2} \frac{\beta_{m_2}}{m_2^s}
\end{align*}
for some $R \in \{1, \dotsc, \lfloor \frac{\log z}{\log z_0} \rfloor\}$ and
\begin{align*}
X^{\varepsilon/2} \leq M_1\leq z \quad \text{and} \quad M_1 M_2 \in (X/(2P_1), 4X/P_1],
\end{align*}
where $z_0=\exp((\log X)/(\log \log X)^3)$, $z=X^{2/11}$, and $\alpha_{m_1}$ are as in~\eqref{eq:alm1def}.

Recall that $\mathcal{U} \subseteq [T_0, T_1] = [X^{1/1000}, X]$. Let us partition $\mathcal{U} = \mathcal{U}_1 \cup \mathcal{U}_2 \cup \mathcal{U}_3$, where
\begin{align*}
\mathcal{U}_1 &:= \{t \in \mathcal{U} \colon |M_1(1+it)| \geq M_1^{-49/206+10\varepsilon} \text{ or } |M_2(1+it)| \geq M_2^{-49/206+10\varepsilon} \}, \\
\mathcal{U}_2 &:= \{t \in \mathcal{U} \colon X^{-1} \leq |M_1(1+it)| < M_1^{-49/206+10\varepsilon} \text{ and } X^{-1} \leq |M_2(1+it)| < M_2^{-49/206+10\varepsilon} \}, \\
\mathcal{U}_3 &:= \{t \in \mathcal{U} \colon |M_1(1+it)| < X^{-1} \text{ or } |M_2(1+it)| < X^{-1} \}.
\end{align*}
Proposition~\ref{prop_typeII} immediately implies that
\[
\int_{\mathcal{U}_1} |M_1(1+it)|^2 |M_2(1+it)|^2\d t \ll \exp(-(\log X)^{1/10}),
\]
since
\begin{align}\label{eq:Mcancel}
\sup_{T_0 \leq t \leq T_1} |M_1(1+it)| \ll \exp(-2(\log X)^{1/10})
\end{align}
as a corollary of the Vinogradov--Korobov zero-free region (see e.g.~\cite[Lemma 1.5]{harman-sieves}). Furthermore trivially
\[
\int_{\mathcal{U}_3} |M_1(1+it)|^2 |M_2(1+it)|^2\d t \ll X \cdot \frac{1}{X^2} \cdot (\log X)^{O(1)} \ll \frac{1}{X^{1/2}}.
\]
Hence we can concentrate on $\mathcal{U}_2$. We split it into $\ll (\log X)^2$ sets of the form
\begin{align}
\label{eq:Us1s2def}
\mathcal{U}_{2, \sigma_1, \sigma_2}:=\{t\in \mathcal{U}_2 \colon |M_1(1+it)|\in (M_1^{-\sigma_1}, 2M_1^{-\sigma_1}], |M_2(1+it)|\in (M_2^{-\sigma_2}, 2M_2^{-\sigma_2}]\}.    
\end{align}
By the definition of $\mathcal{U}_2$, this is non-empty only when $\sigma_j> 49/206-10\varepsilon$ for $j = 1, 2$. In order to deduce~\eqref{eq:Guintclaim}, it now suffices to show that, for any $\sigma_1, \sigma_2 > 49/206-10\varepsilon$, we have
\begin{equation}
\label{eq:Us1s2claim}
|\mathcal{U}_{2, \sigma_1, \sigma_2}| \ll M_1^{2\sigma_1} M_2^{2\sigma_2} \exp(-(\log \log X)^7).
\end{equation}

Let 
\[
k = \left\lfloor \frac{\log T_1}{\log (2P_1)}\right\rfloor
\] 
and $M=P_1^k \in [T_1/(2^kP_1), T_1/2^k]$. Define
\begin{align*}
M(s):=\frac{P_1(s)^k}{k!}=\sum_{m\in \mathcal{M}}\varepsilon_m m^{-s},    
\end{align*}
say, where $\varepsilon_m\in [0,1]$ and $\mathcal{M}$ consists of products of $k$ primes from $(P_1, 2P_1]$. In particular, all $m\in \mathcal{M}$ are $2P_1$-smooth, so that
\begin{align*}
|\mathcal{M}|\leq M^{1-1/a+o(1)}
\end{align*}
by a standard upper bound for the number of smooth numbers (see e.g.~\cite[(1.14)]{ht}). 
Moreover, $\mathcal{M} \subseteq (P_1^k, (2P_1)^k] \subseteq [M,T_1]$ and, by the definition of $\mathcal{U}$,
\begin{align}
\label{eq:Mlow}
|M(1+it)|\geq \frac{M^{-\varepsilon/10}}{k!} \geq M^{-\varepsilon/10-1/a+o(1)}    \textnormal{ for all } t\in \mathcal{U}.
\end{align}

Next we bound $|\mathcal{U}_{2, \sigma_1, \sigma_2}|$ in two different ways. First, by~\eqref{eq:Us1s2def} and~\eqref{eq:Mlow}, we have
\begin{align*}
|\mathcal{U}_{2, \sigma_1, \sigma_2}| \ll M^{\varepsilon/5+2/a+o(1)} M_2^{2 \sigma_2} \int_{\mathcal{U}_{2, \sigma_1, \sigma_2}}|M_2(1+it)|^{2} |M(1+it)|^2 \d t.  
\end{align*}
By Lemma~\ref{le_HB_MVT} applied with $N = M_2$ (and noting that $|\mathcal{M}| \leq T_1^{1/3}$ and $M \in [T_1^{1-\varepsilon^2}, T_1]$), we see that 
\begin{align*}
|\mathcal{U}_{2, \sigma_1, \sigma_2}| &\ll M^{\varepsilon/5+2/a+o(1)} M_2^{2 \sigma_2} \left( \left(\frac{|\mathcal{M}|}{M}\right)^2+T_1^{\varepsilon/10} \frac{|\mathcal{M}|T_1}{M^2 M_2}\right)\\
&\ll M^{\varepsilon/5+o(1)} \left(M_2^{2 \sigma_2} + \frac{T_1^{1+\varepsilon/10} M^{1/a}}{M M_2^{1-2\sigma_2}}\right).   
\end{align*}

Hence~\eqref{eq:Us1s2claim} holds provided that
\begin{align}\label{eq7}
M^{\varepsilon/5} \ll \frac{M_1^{2\sigma_1}}{T_1^{\varepsilon^2}} \quad \text{and} \quad M_1^{1-2\sigma_1}\ll T_1^{-\varepsilon/3} M^{1-1/a}.
\end{align}
Since $M \in [T_1^{1-\varepsilon^2}, T_1]$, $M_1 \geq T_1^{\varepsilon/2}$ and $\sigma_1 > 49/206-10\varepsilon>1/5+2\varepsilon$, the first claim always holds when $\varepsilon$ is sufficiently small. The second claim holds for sufficiently small $\varepsilon$ when
\[
a > \frac{1}{1-\theta(1-2\sigma_1)-\varepsilon},
\]
where we have denoted
\begin{align}\label{eq:theta}
\theta=\frac{\log M_1}{\log X} \in \left[\frac{\varepsilon}{2}, \frac{2}{11}\right].
\end{align}
Hence we can from now on assume that 
\begin{equation}
\label{eq:thsig1}
\theta (1-2\sigma_1) \geq 1-\frac{1}{a}-\varepsilon \quad \text{and} \quad \sigma_1\in \left[\frac{49}{206}-10\varepsilon, \frac{1}{2}\right).
\end{equation}

On the other hand, using~\eqref{eq:Us1s2def} and~\eqref{eq:Mlow} and arguing similarly as before, we see that
\begin{align*}
|\mathcal{U}_{2, \sigma_1, \sigma_2}| &\ll M^{\varepsilon/5+2/a+o(1)} M_1^{10 \sigma_1} \int_{\mathcal{U}_{2, \sigma_1, \sigma_2}}|M_1(1+it)|^{10} |M(1+it)|^2 \d t\\
&\ll M^{\varepsilon/5+o(1)} M_1^{10\sigma_1} \left(1 + \frac{T_1^{1/a + 3\varepsilon}}{M_1^5}\right).   \end{align*}
Since $M_1 \leq T_1^{2/11+\varepsilon/2}$, the second term dominates when $a \leq 11/10$.
Hence~\eqref{eq:Us1s2claim} holds if 
\[
T_1^{1/a+4\varepsilon} \leq M_1^{5-8\sigma_1} M_2^{2\sigma_2}.
\]
Recalling~\eqref{eq:theta} and that $\sigma_2 \geq 49/206-10\varepsilon$, this holds if
\[
(5-8\sigma_1)\theta + (1-\theta)\cdot 2\cdot \frac{49}{206} > \frac{1}{a}+25\varepsilon \iff 4(1-2\sigma_1)\theta + \frac{49}{103} + \theta \left(1-2\cdot \frac{49}{206}\right) > \frac{1}{a}+25\varepsilon.
\]
Now the left-hand side is increasing in $\theta$ (since $\sigma_1 <1/2$), so using~\eqref{eq:thsig1} it suffices to have
\[
4-\frac{4}{a} -4\varepsilon+ \frac{49}{103} + \left(1-\frac{1}{a}-\varepsilon\right) \frac{1-2\cdot \frac{49}{206}}{1-2\sigma_1} > \frac{1}{a}+25\varepsilon.
\]

Since $\sigma_1 \geq 49/206-10\varepsilon$, this holds if
\[
\frac{6}{a} < 5 + \frac{49}{103}-100\varepsilon, 
\]
which in turn holds for $\varepsilon>0$ small enough if
\[
a>\frac{103}{94}+100\varepsilon.
\]
But $\frac{103}{94} < 1.1-1/10000$, and the claim follows.

\begin{remark}\label{rem:gen}
We note that in general if one had Proposition~\ref{prop_typeII} with $\sigma$ in place of $49/206$ and $\theta$ in place of $2/11$ (with $1/5+2\varepsilon\leq \sigma\leq 1/2$), then the first part of the argument in Subsection~\ref{subsec:typeII} would imply that one can deal with type II sums with one variable from $[X^{\varepsilon/2}, X^{\theta}]$ with the exponent
\begin{align}
\label{eq:cgen}
c=1+\frac{1}{1-\theta(1-2\sigma)-\varepsilon}    
\end{align}
in Theorem~\ref{thm_minorant} (with $a\in [c-1-\varepsilon^2,c-1]$ in place of~\eqref{eq:aP1def}).
\end{remark}

\begin{remark}\label{rem:improv}
One could optimize the argument in several ways, but we have decided not to do so, as our relatively clean argument already gives a very substantial improvement over~\cite{tera-primes}. 

For example, one could prove stronger variants of Lemma~\ref{le_laregvalues} inside the set $\mathcal{U}$ by using amplification by $P_1(1+it)^k P_1^{k\varepsilon/10}$ inside the proof of Jutila's large value result and then replacing Jutila's application of fourth moment of zeta by the Deshouillers--Iwaniec theorem (Lemma~\ref{le:WattDI}(ii)) or by Heath-Brown's sparse mean value theorem (Lemma~\ref{le_HB_MVT}). This would lead to a slight improvement of Proposition~\ref{prop_typeII}.

Furthermore, one could obtain better large value results by taking better into account the shape and length of the polynomials $M_1(s)$ and $M_2(s)$. In particular, in the proof of Proposition~\ref{prop_typeII} one could get a better lower bound for large values of $M_1(s)$ since $M_1^\ell \geq T^{5/6} \geq T^{9/11}$, so better large value theorems are available for $M_1(s)^\ell$ than for $M_2(s)$. On the other hand, for $M_2(s)$ it might be of benefit to decompose it further into a product of Dirichlet polynomials and apply large value theorems for its components.
\end{remark}
\section{Results with Heath-Brown's identity} 
\label{se:H-B}
Instead of using Harman's sieve, one could use Heath-Brown's identity. This way the argument would be somewhat simpler, but one would only obtain Theorem~\ref{thm_main} with a somewhat larger interval length $(\log x)^{2+3/13+\varepsilon}$. More precisely, one would obtain
\begin{align}
\label{eq:H-Bresult}
|\{p_1p_2\in (x,x+h]:\, (\log x)^{a}<p_1\leq (\log x)^{a+\varepsilon/2}\}|\gg_{\varepsilon} \frac{h}{\log x}    
\end{align}
for all but $\ll X/(\log X)^{\delta}$ integers $x\in [2,X]$, with $h = (\log x)^{2+3/13+\varepsilon}$ and $a=1+3/13$. On the other hand, assuming the Lindel\"of hypothesis, one can use Heath-Brown's identity to obtain $h = (\log x)^{2+\varepsilon}$ and $a=1+\varepsilon/2$.

We sketch the proofs here: It suffices to show that, with $P_1\in ((\log X)^{a},(\log X)^{a+\varepsilon}]$, $h_1 = X^{99/100}$,  $h=(\log X)^{a+1}$ and $a$ as in one of the above claims,
\begin{align*}
\frac{1}{X}\int_{X}^{2X}\left|\frac{1}{h}\sum_{\substack{x < p_1n\leq x+h\\p_1 \sim P_1}}\Lambda(n)-\frac{1}{h_1}\sum_{\substack{x < p_1n\leq x+h_1\\p_1 \sim P_1}}\Lambda(n)\right|^2 \d x \ll\frac{1}{(\log X)^{\varepsilon}},
\end{align*}
where $\Lambda(n)$ is the von Mangoldt function. We reduce to mean squares of Dirichlet polynomials as in Lemma~\ref{lem:DirPolRed} (but with $\Lambda$ in place of $\rho^{-}$ and $(\log X)^{-\varepsilon}$ in place of $(\log X)^{-2-\varepsilon}$) and handle $\int_{\mathcal{T}}$ as in Section~\ref{sec:proof}. 

Let $L \in \mathbb{N}$ be fixed and $Y = X/P_1$. Applying Heath-Brown's identity~\cite[Proposition 13.37]{iw-kow} and splitting the variables into short intervals gives a set $\mathcal{F}$ consisting of $(\log X)^{O_{A, L}(1)}$ functions $f \colon \mathbb{N} \to \mathbb{C}$ such that
\[
\Lambda(n) 1_{n \in (Y/2, 4Y]} = \sum_{f \in \mathcal{F}} f(n) + c_n
\]
where $c_n$ is as in Proposition~\ref{prop_rho} and each $f$ is of the form
\[
f = a^{(1)} \ast \dotsb \ast a^{(\ell)}
\]
for some $\ell \leq 2L$ with each $a^{(i)}(n)$ one of $1_{(N_i, (1+\delta_A)N_i]} \log n$, $1_{(N_i, (1+\delta_A)N_i]}$ or $1_{(N_i, (1+\delta_A)N_i]}\mu(n)$. Moreover $N_1 \dots N_\ell \asymp Y$ and, for each $i$ with $a^{(i)}(n) = 1_{(N_i, (1+\delta_A)N_i]}\mu(n)$, we have $N_i \ll Y^{1/L}$. 

\subsection{Unconditional result with \texorpdfstring{$h = (\log x)^{2+3/13+\varepsilon}$}{h=(log x)2+3/13+eps}}
For the unconditional result we choose $L = 3$ in Heath-Brown's identity. If $X^{\varepsilon/10} \leq N_j\leq X^{1/3}$ for some $j$, then $f(n)$ is a type II sum
\begin{align*}
\sum_{\substack{n=m_1m_2\\M_1 < m_1 \leq (1+\delta_A) M_1 \\ m_2 \sim M_2}} \alpha_{m_1} \beta_{m_2}  
\end{align*}
with $\alpha_{m_1}$ either $1$, $\log m_1$ or $\mu(m_1)$ and
\begin{align*}
X^{\varepsilon/10} \leq M_1\leq X^{1/3} \quad \text{and} \quad M_1 M_2 \in (X/(2P_1), 4X/P_1].
\end{align*}
Otherwise the product of two largest $N_j$ must be $\gg X^{1-4\varepsilon/10}/P_1$ and hence we have a type I sum
\begin{align*}
\sum_{\substack{n=m_1m_2\\ m_1 \sim M_1 \\ M_2 < m_2 \leq (1+\delta_A) M_2}}\alpha_{m_1} \quad \text{or} \quad \sum_{\substack{n=m_1m_2\\ m_1 \sim M_1 \\ M_2 < m_2 \leq (1+\delta_A) M_2}}\alpha_{m_1} \log m_2  
\end{align*}
with $M_1 \leq X^{1/2+\varepsilon}$ and $M_1 M_2 \in (X/(2P_1), 4X/P_1]$.
The type I sums can be handled as before (using partial summation in the second case).

For type II sums we argue similarly to Section~\ref{subsec:typeII}, but use a variant of Lemma~\ref{le_laregvalues}, where $T^{1-\varepsilon/5} \geq N \geq T^{2/3}$ and $5\varepsilon \leq \sigma \leq 7/32-10\varepsilon$. Such a variant follows from Bourgain's zero density estimate~\cite[Lemma 4.60]{bourgain}. The coefficients $\alpha_{m_1}$ are of different shape than previously. However, this is not an issue since $M_1(s)$ still satisfies~\eqref{eq:Mcancel} and furthermore, since $M_1 \geq X^{\varepsilon/10}$, in the proof of the variant of Proposition~\ref{prop_typeII} the parameter $\ell_1$ is bounded, so that the coefficients of $M_1(s)^{\ell_1}$ are divisor-bounded.

Now we have in~\eqref{eq:cgen} $\theta = 1/3$ and $\sigma = 7/32-10\varepsilon$ which, after adjusting $\varepsilon$, gives Theorem~\ref{thm_minorant} with $c = 2+3/13+\varepsilon$. and $a\in [c-1-\varepsilon/2,c-1]$

\subsection{The Lindel\"of hypothesis implies \texorpdfstring{$h = (\log x)^{2+\varepsilon}$}{h=(log x)2+eps}}
\label{subsec:Lindelof}
To obtain the result under the Lindel\"of hypothesis, we apply Heath-Brown's identity with $L = \lceil 1/\varepsilon \rceil$. Now if $X^{\varepsilon/(10L)} \leq N_j \leq X^{\varepsilon}$ for some $j$, then $f(n)$ is a type II sum
\begin{align*}
\sum_{\substack{n=m_1m_2\\M_1 < m_1 \leq (1+\delta_A) M_1 \\ m_2 \sim M_2}} \alpha_{m_1} \beta_{m_2}  
\end{align*}
with $\alpha_{m_1}$ either $1$, $\log m_1$ or $\mu(m_1)$ and
\begin{align*}
X^{\varepsilon/(10L)} \leq M_1\leq X^{\varepsilon} \quad \text{and} \quad M_1 M_2 \in (X/(2P_1), 4X/P_1].
\end{align*}
In this case we have in~\eqref{eq:cgen} $\theta = \varepsilon$ (and can take e.g. $\sigma = 7/32-10\varepsilon$) which, after adjusting $\varepsilon$, gives $c = 2+\varepsilon$.

Hence we can assume that all the factors longer than $X^{\varepsilon/(10L)}$ have coefficients $1$ or $\log$. Now if $N_j \geq X^{1/2-\varepsilon}$ for some $j$, then we have a type I sum which can be dealt with as before.

In the remaining case we have, for some $\ell \in \{2, \dotsc, 2L\}$, a type $I_{\ell}$ sum of the form
\begin{align*}
\sum_{\substack{n=k m_1m_2 \dotsm m_\ell \\ k \sim K \\M_j < m_j \leq (1+\delta_A) M_j}}\alpha_{k} \beta_{m_1} \dotsc \beta_{m_\ell} 
\end{align*}
with $K \leq X^{\varepsilon/5}$, $M_1, \dotsc, M_\ell \in (X^\varepsilon, X^{1/2-\varepsilon}], K M_1 M_2 \dotsm M_\ell \in (X/(2P_1), 4X/P_1]$, and $\beta_{m_j} \in \{1, \log m_j\}$. 

Under the Lindel\"of hypothesis we have, for $j = 1, \dotsc, \ell$, 
\[
\left|\sum_{M_j< m_j \leq (1+\delta_A)M_j }\frac{\beta_{m_j}}{m_j^{1+it}} \right| \ll_{\varepsilon} \frac{(|t|+1)^{\varepsilon^2/10}}{M_j^{1/2}}+\frac{1}{|t|}
\]
(this follows e.g. from~\cite[(9.21)]{iw-kow} and partial summation). For those $t$ for which the second term dominates for some $j$, we can bound all the other Dirichlet polynomials trivially, obtaining a contribution of 
\[
\int_{T_0}^{T_1} \frac{(\log X)^{O(1)}}{|t|^2}\d t \ll \frac{1}{T_0^{1/2}}.
\]
Otherwise we essentially have in~\eqref{eq:cgen} $\theta \in (\varepsilon, 1/2-\varepsilon]$ and $\sigma = 1/2-\varepsilon$ which, after adjusting $\varepsilon$, gives again Theorem~\ref{thm_minorant} with $c = 2+\varepsilon$. and $a\in [c-1-\varepsilon/2,c-1]$.

\section{All intervals} \label{sec:allintervals}

We lastly turn to the proof of Theorem~\ref{thm_all}, which closely follows the proof of Theorem~\ref{thm_main}. 

Let $a = 1.1$, $c=1+a/2$, $h=\sqrt{x}(\log x)^c$, and $P_1=(\log x)^{a}$. We will show that
\begin{align*}
\left|\left\{p_1p_2p_3\in (x,x+h]:\, p_1\sim P_1,p_2\sim \frac{\sqrt{x/P_1}}{2}\right\}\right|\gg \frac{h}{(\log P_1)(\log x)^2}.
\end{align*}

Let $\rho^{-}$ be the same minorant function as in Theorem~\ref{thm_minorant} but with $\sqrt{x P_1}$ in place of $X$, i.e. $\rho^-$ is defined by~\eqref{eq:rho-def} with $X = \sqrt{xP_1}$ and $z = \sqrt{x P_1}^{2/11}$. Then it suffices to show that
\begin{align*}
\sum_{\substack{x < p_1p_2n\leq x+h\\p_1\sim P_1\\p_2\sim \sqrt{x/P_1}/2}}\rho^{-}(n)\gg \frac{h}{(\log P_1)(\log x)^2}.      
\end{align*}
By a slight variant of Theorem~\ref{thm_minorant}(ii) which is proved in the same way, this reduces to showing that
\begin{align*}
\frac{1}{h}\sum_{\substack{x < p_1p_2n\leq x+h\\p_1\sim P_1\\p_2\sim \sqrt{x/P_1}/2}}\rho^{-}(n)=\frac{1}{h_1}\sum_{\substack{x < p_1p_2n\leq x+h_1\\p_1\sim P_1\\p_2\sim \sqrt{x/P_1}/2}}\rho^{-}(n)+o\left(\frac{1}{(\log P_1)(\log x)^2}\right),
\end{align*}
where $h_1=x^{99/100}$. Next we sketch the standard deduction to mean values of Dirichlet polynomials. Write 
\begin{align*}
P_1(s):=\sum_{p_1\sim P_1}p_1^{-s},\quad P_2(s):=\sum_{p_2\sim \sqrt{x/P_1}/2}p_2^{-s},\quad P(s):=\sum_{\sqrt{x/P_1}/2 < n\leq 4\sqrt{x/P_1}}\rho^{-}(n)n^{-s},
\end{align*}
and $\alpha(s) = P_1(s) P_2(s) P(s)$. Using Perron's formula (see e.g.~\cite[Corollary 5.3]{MVbook}, noting that the coefficients of $\alpha(s)$ are bounded and supported on $m \asymp x$), and dealing with the integral over $[-T_0, T_0]$ similarly to e.g.~\cite[Proof of Lemma 7.2]{harman-sieves}), it suffices to show, for some small $\varepsilon>0$,
\begin{align*}
\frac{1}{h}\int_{T_0}^{x (\log x)^{10}}|P_1(1+it)||P_2(1+it)||P(1+it)| \min\left\{\frac{x}{|t|}, h\right\} \d t \ll \frac{1}{(\log x)^{2+\varepsilon/2}},    
\end{align*}
where $T_0=x^{1/1000}$.
Considering separately the integral over $[T_0, x/h]$ and splitting the remaining integral over $(x/h, x(\log x)^{10}]$ dyadically into $\ll \log \log x$ integrals, we see that it suffices to show that, for any $T \in [x/h, x(\log x)^{10}]$, we have
\begin{align}\label{eq:p1p2p2}
\int_{T_0}^{T}|P_1(1+it)||P_2(1+it)||P(1+it)| \d t \ll \frac{T}{x/h} \cdot \frac{1}{(\log x)^{2+\varepsilon}}.    
\end{align}

Note that by the improved mean value theorem (Lemma~\ref{le:contMVT2}) we have, for $T \geq x/h$,
\begin{align*}
\int_{T_0}^{T}|P_2(1+it)|^2\d t\ll \frac{T}{\sqrt{x/P_1}\log x}+\frac{1}{(\log x)^2}\ll \frac{T}{x/h}\cdot \frac{1}{(\log x)^2}
\end{align*}
since $h=\sqrt{x P_1}\log x$. Hence, applying the Cauchy--Schwarz inequality to \eqref{eq:p1p2p2}, it suffices to prove that, for any $T \geq x/h$,
\begin{align*}
\int_{T_0}^{T}|P_1(1+it)|^2|P(1+it)|^2\d t\ll \frac{T}{x/h} \frac{1}{(\log x)^{2+2\varepsilon}}.     
\end{align*}
But since $a = 1.1$, this is essentially the claim~\eqref{eq:dirpol} that was proved in Section~\ref{sec:proof} (after adjusting $\varepsilon$ and replacing $X$ with $\sqrt{xP_1}$).

\bibliography{refs}
\bibliographystyle{plain}
\end{document}